\documentclass[oneside,american]{amsart}
\usepackage[T1]{fontenc}
\usepackage[utf8]{inputenc}
\usepackage{geometry}
\usepackage{color}
\usepackage{verbatim}
\usepackage{amstext}
\usepackage{amsthm}
\usepackage{amssymb}
\usepackage{graphicx}
\usepackage{esint}

\makeatletter
\numberwithin{equation}{section}
\numberwithin{figure}{section}
\theoremstyle{plain}
\newtheorem{thm}{\protect\theoremname}
  \theoremstyle{plain}
  \newtheorem{cor}[thm]{\protect\corollaryname}
  \theoremstyle{plain}
  \newtheorem{lem}[thm]{\protect\lemmaname}
  \theoremstyle{remark}
  \newtheorem{rem}[thm]{\protect\remarkname}
  \theoremstyle{plain}
  \newtheorem{prop}[thm]{\protect\propositionname}
  \theoremstyle{plain}
\newtheorem{defn}{\protect\definitionname}
 \theoremstyle{plain}
\newtheorem{ex}{\protect\examplename}

\usepackage{todonotes}

\makeatother

\usepackage{babel}
  \providecommand{\corollaryname}{Corollary}
  \providecommand{\lemmaname}{Lemma}
  \providecommand{\propositionname}{Proposition}
  \providecommand{\remarkname}{Remark}
\providecommand{\theoremname}{Theorem}
\providecommand{\definitionname}{Definition}
\providecommand{\examplename}{Example}

\begin{document}

\title{Hyperbolic polynomials and linear-type generating functions}

\author{Tam\'as Forg\'acs \and Khang Tran}

\maketitle
\global\long\def\Re{\operatorname{Re}}

\global\long\def\Im{\operatorname{Im}}

\global\long\def\Arg{\operatorname{Arg}}

\global\long\def\Log{\operatorname{Log}}

\section{Introduction}

The problem of describing the zero distribution of a sequence of polynomials remains an active area of research. From the classical methods of orthogonality to spectral theory of positive matrices asymptotic descriptions (i.e. approximate locations) there is a plethora of approaches to the problem. Naturally, depending on the approach, the methods used to investigate the problem can be quite different. In this paper we follow the approach employed in the works \cite{ft}, \cite{ft-1}, \cite{tran}, \cite{tz} and \cite{tz-1} as we analyze the zero location of a sequence of polynomials $\{H_m(z) \}_{m=1}^{\infty}$ generated by the relation 
\[
\sum_{m=0}^{\infty}H_{m}(z)t^{m}=\frac{1}{P(t)+zt^{r}Q(t)}, \qquad (r \in \mathbb{N})
\]
where $P$ and $Q$ are real stable polynomials with some restrictions on their zero locus. In two of our recent papers considering such problems  the choice of the generating functions was largely motivated by the theory of multiplier sequences (and stability preserving linear operators in general). As such, we considered the (family of) generating functions
\begin{equation} \label{eq:genfunctions}
 \frac{1}{(1-t)^n+zt^r} \qquad \text{and} \qquad \frac{1}{P(t)+zt^r},
\end{equation}
where $P(t)$ is a polynomial with only positive zeros, and showed that the sequence of polynomials generated by these functions is eventually hyperbolic (see \cite[p.632, \,Theorem 1]{ft} and \cite[p.619, \,Theorem 1]{ft-1}). Establishing that \textit{all} polynomials generated by functions of the type in \eqref{eq:genfunctions} have only real zeros -- not just the ones far enough out in the sequence -- remains and open problem. Its resolution (in the positive) is in fact quite desirable, as it would open up the avenue to extending the family of functions that generate hyperbolic polynomials using locally uniform approximation arguments. This paper generalizes the results in \cite{ft} and \cite{ft-1} by considering generating functions whose denominators are elements of $\mathbb{R}[t][z]$ with coefficient polynomials that are hyperbolic. The elements of $\mathbb{R}[t][z]$ we consider can be viewed as linear combinations of $1$ and $z$ with coefficients that are hyperbolic polynomials. In this light, connections between the properties of the coefficient polynomials and the stability of the generated sequence emerge, very much in the flavor of classical stability theory \`a la Hermite-Biehler.\\
\indent The Hermite-Biehler theorem (see for example \cite[p.\,197]{rs}) and some of its generalizations (\cite{es}, \cite{hdb1}, \cite{hdb2}) address the connections between the stability of a polynomial $f=p(x^2)+xq(x^2)$, and the interlacing of the zeros of its `constituents' $p$ and $q$. The cited generalizations study the extent to which one may still be able to draw conclusions about the location of the zeros of the constituents, even if the polynomial $f$ is not Hurwitz-stable. In particular, if $n_-$ (resp. $n_+$) denote the number of zeros of a polynomial $f$ in the left (resp. right) half plane, then the number of (interlacing) real zeros of its constituents is bounded below in terms of $|n_{-}-n_+|$. 
\newline
\indent If we regard $P(t)$ and $Q(t)$ as the 'constituents' of the polynomials $H_m(z)$, it would be reasonable to expect -- analogously to the Hermite--Biehler theory -- that the interlacing of the zeros of $P$ and $Q$ would imply hyperbolicity (real stability) of the polynomials $H_m(z)$. Alas, somewhat the contrary is true: the more separated the zeros of $P$ and $Q$ are, the 'better' in terms of the hyperbolicity of the generated sequence (c.f. Corollary \ref{cor:separatedzeros} and Remark \ref{rem:negzeroP}). \\
\indent The rest of the paper is organized as follows. Section \ref{sec:mainresult} contains the setup and statement of the main result. Section \ref{sec:tauz} is devoted to the development of two key functions $\tau(\theta)$ and $z(\theta)$, which allow us to identify points in the interval $(0, \pi/r)$ with the zeros of our generated polynomials $H_m(z)$ in a one-to-one fashion. In Section \ref{sec:completetheproof} we establish the stability of the polynomials $H_m(z)$ and complete the proof of Theorem \ref{thm:maintheorem} modulo three auxiliary lemmas. We prove these lemmas in the concluding section of the paper.
\section{The main result}\label{sec:mainresult}
Let 
\begin{equation}\label{eq:prodreppolys}
P(t)=\prod_{-p_{-}<k\le p_{+}}(t-\tau_{k}),\qquad\text{ and }\qquad Q(t)=\prod_{-q_{-}<k\le q_{+}}(t-\gamma_{k})
\end{equation}
be two hyperbolic polynomials with $p_{+}$, $q_{+}$ positive and $p_{-}$,
$q_{-}$ negative zeros respectively, and suppose that $P(0),Q(0)\ne0$. 
We arrange the zeros of $P(t)$ and $Q(t)$ in an increasing order
according to their indices. In particular, $\tau_{0}$ ($\gamma_0$ resp.) is
the largest negative zero of $P(t)$ ($Q(t)$ resp.), while $\tau_{1}$ ($\gamma_1$ resp.) is its smallest positive zero. 
\begin{defn} \label{def:zerocount} Let $P$ and $Q$ be polynomials. For each $x>0$, we let $n_{+}^{P}(x)$ and $n_{+}^{Q}(x)$ be the number of positive zeros of $P(t)$ and $Q(t)$ on $(0,x]$ counting multiplicity. Similarly, for each $x<0$, we let $n_{-}^{P}(x)$ and $n_{-}^{Q}(x)$ be the number of negative zeros of $P(t)$ and $Q(t)$ on $[x,0)$.
\end{defn}

In light of the Hermite-Biehler extensions discussed above, it is perhaps not surprising that the quantities $n_{+}^P-n_{+}^Q$ and $n_{-}^Q-n_{-}^P$ appear in what follows, as those controlling the extent to which the zeros of the constituents of the sequence $H_m(z)$ are allowed to intermingle without destroying the hyperbolicity of $H_m(z)$.

\begin{lem}
\label{lem:taexistence} Let $P, Q$ be as in \eqref{eq:prodreppolys}, $n_+^P$ and $n_+^Q$ be as in Definition \ref{def:zerocount}, and let 
\[
R(t)=r-\frac{tP'(t)}{P(t)}+\frac{tQ'(t)}{Q(t)}.
\] 
Suppose that 
\begin{itemize}
\item[(i)] $n_{+}^{P}(x)-n_{+}^{Q}(x)\ge2$ $\forall x\ge\tau_{2}$,
and $n_{+}^{Q}(x)=0$, $\forall x\in(0,\tau_{2}]$ and
\item[(ii)] $\Im R(t)>0$ on the sector $\{t \ \big| \ 0<|t|<\tau_{2},0<\Arg t<\pi/r\}$. 
\end{itemize} 
Under these assumptions $P(t)R(t)$ has a unique zero $t_{a}$
on $(\tau_{1},\tau_{2})$ (with multiplicity possibly bigger than
one) if $\tau_{1}<\tau_{2}$, and $t_{a}=\tau_{1}=\tau_{2}$ if $\tau_{1}=\tau_{2}$. Furthermore, $t_{a}$ is the smallest positive zero of $P(t)R(t)$. 
\end{lem}

\begin{proof}
We note that for any $t\ne\tau_{k},\gamma_{j}$, $-q_-<j\leq q_+, -p_-<k\leq p_+$, we
have 
\begin{align}
\Im R(t) & =\Im\left(\sum_{-q_{-}<j\le q_{+}}\frac{t}{t-\gamma_{j}}-\sum_{-p_{-}<k\le p_{+}}\frac{t}{t-\tau_{k}}\right)\nonumber \\
 & =\Im\left(-\sum_{-q_{-}<j\le q_{+}}\frac{t\gamma_{j}}{|t-\gamma_{j}|^{2}}+\sum_{-p_{-}<k\le p_{+}}\frac{t\tau_{k}}{|t-\tau_{k}|^{2}}\right)\nonumber \\
 & =\left(\sum_{-p_{-}<k\le p_{+}}\frac{\tau_{k}}{|t-\tau_{k}|^{2}}-\sum_{-q_{-}<j\le q_{+}}\frac{\gamma_{j}}{|t-\gamma_{j}|^{2}}\right)\Im t.\label{eq:ImRexp}
\end{align}
Thus condition (ii) implies that 
\[
\sum_{-p_{-}<k\le p_{+}}\frac{\tau_{k}}{|t-\tau_{k}|^{2}}-\sum_{-q_{-}<j\le q_{+}}\frac{\gamma_{j}}{|t-\gamma_{j}|^{2}}>0
\]
for all $t\in\{t|0<|t|<\tau_{2},0<\Arg t<\pi/r\}$. We let $t$ approach
the $x$-axis within this sector to conclude that 
\begin{equation}
R'(t)=\sum_{-p_{-}<k\le p_{+}}\frac{\tau_{k}}{(t-\tau_{k}){}^{2}}-\sum_{-q_{-}<j\le q_{+}}\frac{\gamma_{j}}{(t-\gamma_{j})^{2}}\ge0,\qquad\forall t\in[0,\tau_{2})\backslash\{\tau_{1}\}.\label{eq:Rprime}
\end{equation}
If $\tau_{1}<\tau_{2}$, the result now follows since $R(0)=r>0$,
$R(t)$ is increasing on $(0,\tau_{1})\cup(\tau_{1},\tau_{2})$, $\lim_{t\rightarrow\tau_{1}^{+}}R(t)=-\infty$
and $\lim_{t\rightarrow\tau_{2}^{-}}R(t)=\infty$. If $\tau_{1}=\tau_{2}$,
then $P(\tau_{1})=P'(\tau_{1})=0$ and consequently 
\[
P(\tau_{1})R(\tau_{1})=rP(\tau_{1})-tP'(\tau_{1})+\frac{P(\tau_{1})Q'(\tau_{1})}{Q(\tau_{1})}=0.
\]
\end{proof}
Before we state our main result (c.f. Theorem \ref{thm:maintheorem}) we need to make one more definition.
\begin{defn} Given a polynomial $P(z)$, we denote by $\mathcal{Z}(P(z))$ the set of zeros of $P(z)$.
\end{defn}

\begin{thm}
\label{thm:maintheorem}Let $P,Q$ be polynomials as in \eqref{eq:prodreppolys}, and the functions $n_-^P, n_+^P, n_-^Q,n_+^Q$ be as in Definition \ref{def:zerocount}. Consider the sequence of polynomials $\left\{ H_{m}(z)\right\} _{m=0}^{\infty}$ generated by the relation 
\[
\tag{$\ddagger$} \qquad \sum_{m=0}^{\infty}H_{m}(z)t^{m}=\frac{1}{P(t)+zt^{r}Q(t)}=\frac{1}{D(t,z)},\qquad r\ge2.
\]
If
\begin{enumerate}
\item \label{cond1}$n_{+}^{P}(x)-n_{+}^{Q}(x)\ge2$, $\forall x\ge\tau_{2}$,
and $n_{+}^{Q}(x)=0$, $\forall x\in(0,\tau_{2}]$, 
\item \label{cond2}$n_{-}^{Q}(x)-n_{-}^{P}(x)\ge0$, $\forall x<0$, 
\item \label{cond3}$\Im R(t)>0$ on the sector $\{t|0<|t|<\tau_{2},0<\Arg t<\pi/r\}$, 
\item \label{cond4}$\Im R(t)>0$ on the semi-disk $\{t|0<|t|<t_{a},0<\Arg t<\pi\},$ 
\end{enumerate}
then the zeros of $H_{m}(z)$ are real and of the same sign $(-1)^{p_{+}-q_{+}}$ for all $m \gg 1$.
Moreover, ${\displaystyle \bigcup_{m\gg 1}^{\infty}\mathcal{Z}(H_{m})}$
is dense between $a={\displaystyle -\frac{P(t_{a})}{t_{a}^{r}Q(t_{a})}}$
and $(-1)^{p_{+}-q_{+}}\infty$. 
\end{thm}

\begin{cor}[to the proof of Lemma \ref{lem:taexistence}] \label{cor:separatedzeros} Consider the generating relation $(\ddagger)$. If the zeros of $P(t)$ are positive and those of $Q(t)$ are negative, then the zeros of $H_{m}(z)$ are real for all $m \gg 1$.
\end{cor}
\begin{proof} It is straightforward that under the assumption of the corollary, conditions (1) and (2) of Theorem \ref{thm:maintheorem} are satisfied. In addition, equation (\ref{eq:ImRexp}) guarantees that conditions (3) and (4) are also satisfied. The result follows. 
\end{proof}
\begin{ex} If $P(t)=(t+2)(t-1)(t-2)(t-4)$, $Q(t)=(t+1)(t-3)(t-5)$ and $r=3$, then $t_a \approx 1.3$, and conditions $(3)$ and $(4)$ are satisfied, as illustrated in Figure \ref{fig:example}. 
\begin{figure}[htbp]
\begin{center}
\includegraphics[scale=0.6]{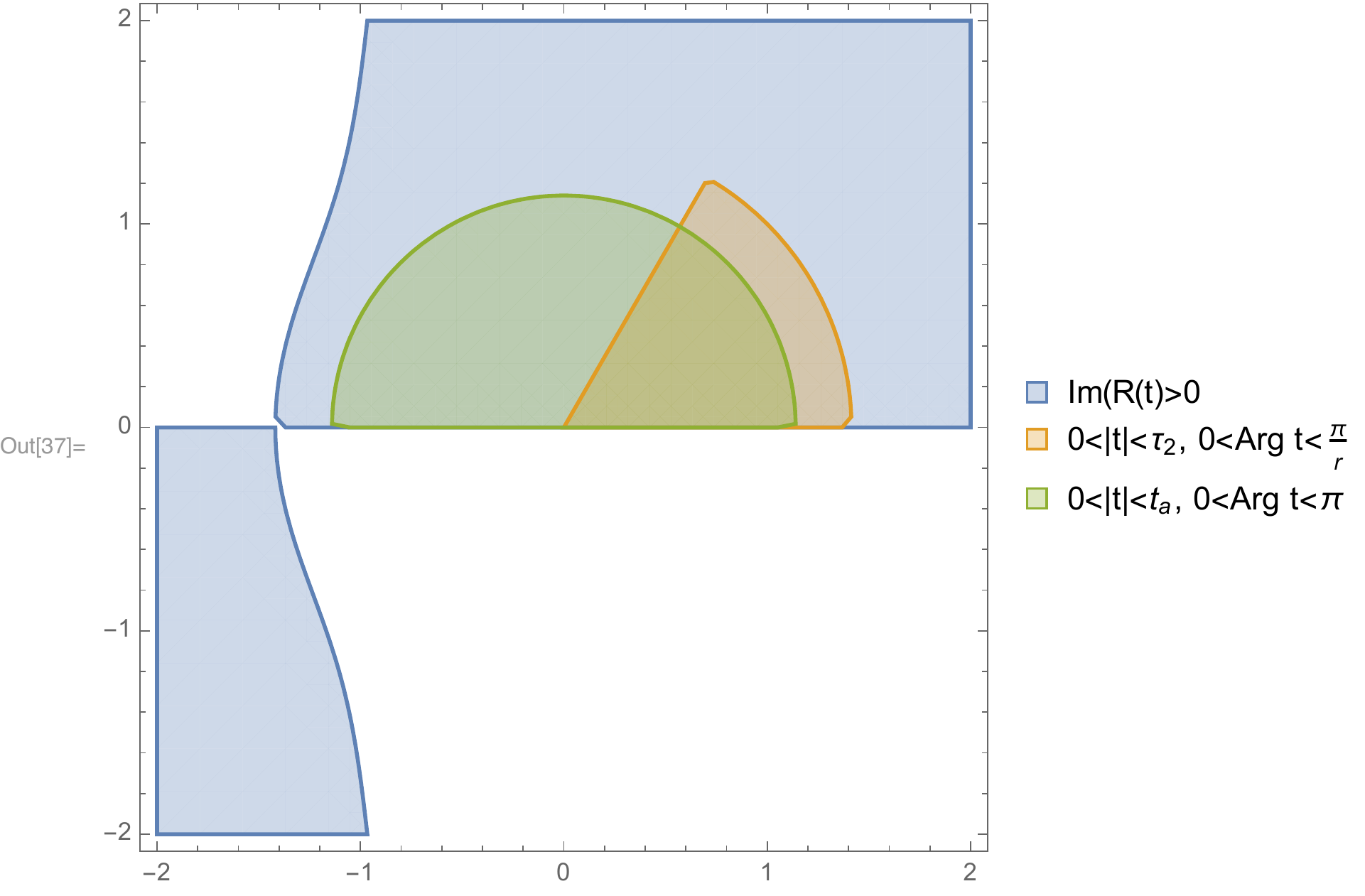}
\caption{The relevant regions for the applicability of Theorem \ref{thm:maintheorem} for the choices $P(t)=(t+2)(t-1)(t-2)(t-4)$ and $Q(t)=(t+1)(t-3)(t-5)$}
\label{fig:example}
\end{center}
\end{figure}
We thus conclude that the polynomials generated by $\displaystyle{\frac{1}{P(t)+zt^3Q(t)}}$ are all hyperbolic for $m \gg 1$.
\end{ex}
\begin{rem} \label{rem:negzeroP} The following observations are immediate:
\begin{itemize}
\item[(i)] As $t\to\tau_{0}$ in the upper half plane,
the right hand side of \eqref{eq:ImRexp} eventually turns negative.
It follows from condition \eqref{cond4} in Theorem \ref{thm:maintheorem}
that $|\tau_{0}|>t_{a}$. 
\item[(ii)] The conclusion of Theorem \ref{thm:maintheorem} is false if we allow the zeros of $P$ and $Q$ to interlace. For example, with $P(t)=(t-1)(t-3)(t-5)$ and $Q(t)=(t-2)(t-4)$ and $r=3$ we see that $H_{16}(z)$ has a non-real root $z=-0.58844...+i \cdot 0.106817... $.
\item[(iii)] $\deg H_m(z) \leq \lfloor m/r \rfloor$ for all $m \geq 0$. This is most readily deduced from the equations
\[
\left(P(\Delta)+z\Delta^rQ(\Delta)\right)[H_m(z)]=\left\{\begin{array}{cc} 1 & m=0 \\ 0 & m\geq 1\end{array} \right.
\]
where $\Delta[H_m(z)]=H_{m-1}(z)$, and $H_{-k}(z) \equiv 0$ for $k \in \mathbb{N}$.
\item[(iv)] With the substitution $t$ by $-t$ we see that the zeros of $H_{m}(z)$
are still real if the zeros of $P(t)$ are negative and $Q(t)$ are
positive. 
\end{itemize}
\end{rem}

\section{The functions $\tau(\theta)$ and $z(\theta)$} \label{sec:tauz}

For each $t=\tau e^{i\theta}$, $0<\theta<\pi$, we define the angles
$0<\theta_{k}(t),\eta_{j}(t)<\pi$ implicitly by 
\begin{eqnarray}
\frac{\tau e^{i\theta}-\tau_{k}}{\tau e^{-i\theta}-\tau_{k}} & = & e^{2i\theta_{k}(t)}\qquad(-p_{-}<k\leq p_{+})\label{eq:thetadef}\\
\frac{\tau e^{i\theta}-\gamma_{j}}{\tau e^{-i\theta}-\gamma_{j}} & = & e^{2i\eta_{j}(t)}\qquad(-q_{-}<j\leq q_{+}).\label{eq:etadef}
\end{eqnarray}
From these equations we obtain 
\begin{equation}
\tau=\tau_{k}\frac{\sin\theta_{k}(t)}{\sin(\theta_{k}(t)-\theta)}=\gamma_{j}\frac{\sin\eta_{j}(t)}{\sin(\eta_{j}(t)-\theta)}\qquad(-p_{-}<k\leq p_{+},\quad-q_{-}<j\leq q_{+}).\label{eq:taudef}
\end{equation}
\begin{figure}
\begin{centering}
\includegraphics[scale=0.7]{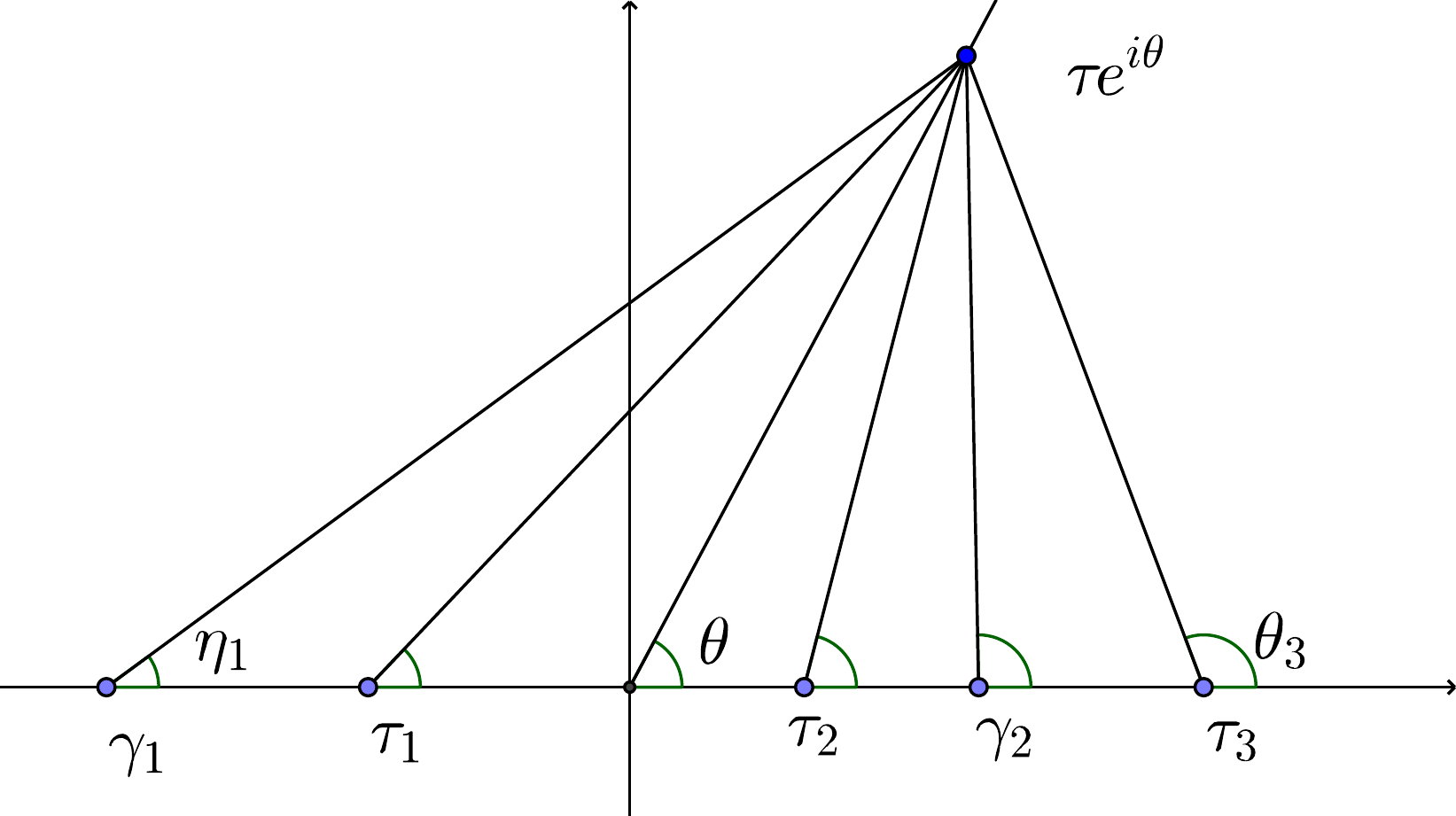} 
\par\end{centering}
\caption{\label{fig:angledef}The angles $\theta_{k}(t)$ and $\eta_{k}(t)$}
\end{figure}

Let $\Log(t)$ denote the principal branch of the logarithm. Then
the function 
\begin{equation}
f(t)=r\Log t+\sum_{-q_{-}<j\leq q_{+}}\Log(t-\gamma_{j})-\sum_{-p_{-}<k\leq p_{+}}\Log(t-\tau_{k})\label{eq:logfunc}
\end{equation}
is analytic on the region $\Im t>0$, and hence $f$ satisfies the
Cauchy-Riemann equations there: 
\begin{align}
\tau\frac{\partial\Re f}{\partial\tau} & =\frac{\partial\Im f}{\partial\theta}=\Re R(t),\nonumber \\
\frac{\partial\Re f}{\partial\theta} & =-\tau\frac{\partial\Im f}{\partial\tau}=-\Im R(t).\label{eq:CauchyRiemann}
\end{align}
On the other hand, 
\begin{align*}
\sum_{-p_{-}<k\le p_{+}}\theta_{k}(t)-\sum_{-q_{-}<k\le q_{+}}\eta_{k}(t)-r\theta & =-\Im f, \qquad \textrm{and}\\
\ln\left|\frac{t^{r}Q(t)}{P(t)}\right| & =\Re f.
\end{align*}
We thus arrive at the following lemmas. 
\begin{lem}
\label{lem:ImR}Suppose $t=\tau e^{i\theta}$, $0<\theta<\pi$. The following statements are equivalent 
\begin{enumerate}
\item $\Im R(t)>0$. 
\item For any fixed $\theta$, the function 
\[
\sum_{-p_{-}<k\le p_{+}}\theta_{k}(t)-\sum_{-q_{-}<j\le q_{+}}\eta_{j}(t)
\]
is strictly decreasing in $\tau$. 
\item For any fixed $\tau$, the function 
\[
\left|\frac{t^{r}Q(t)}{P(t)}\right|
\]
is strictly decreasing in $\theta\in(0,\pi)$. 
\end{enumerate}
\end{lem}

\begin{lem}
\label{lem:ReR}Suppose $t=\tau e^{i\theta}$, $0<\theta<\pi$. The following statements are equivalent 
\begin{enumerate}
\item $\Re R(t)>0$. 
\item For any fixed $\tau$, the function 
\[
\sum_{-p_{-}<k\le p_{+}}\theta_{k}(t)-\sum_{-q_{-}<k\le q_{+}}\eta_{k}(t)-r\theta
\]
is strictly decreasing in $\theta$ . 
\item For any fixed $\theta$, the function 
\[
\left|\frac{t^{r}Q(t)}{P(t)}\right|
\]
is strictly decreasing in $\tau$. 
\end{enumerate}
\end{lem}

We now consider the special $t=\tau e^{i\theta}$ for which the following
equality holds: 
\begin{equation}
\sum_{-p_{-}<k\le p_{+}}\theta_{k}(t)-\sum_{-q_{-}<j\le q_{+}}\eta_{j}(t)-r\theta=(p_{+}-q_{+}-1)\pi.\label{eq:sumangles}
\end{equation}
\begin{lem}
\label{lem:tauupperbound}If $t=\tau e^{i\theta}$, $0<\theta<\pi/r$,
and \eqref{eq:sumangles} holds, then $\tau<\tau_{2}$. 
\end{lem}

\begin{proof}
Since $n_{-}^{Q}(x)\le n_{-}^{P}(x)$, $\forall x<0$, we deduce that
\[
\sum_{-p_{-}<k\le0}\theta_{k}-\sum_{-q_{-}<j\le0}\eta_{j}\le0,
\]
from which \eqref{eq:sumangles} implies 
\begin{equation}
(p_{+}-q_{+}-1)\pi\le\sum_{k=1}^{p_{+}}\theta_{k}-\sum_{j=1}^{q_{+}}\eta_{j}-r\theta.\label{eq:diffplusangles}
\end{equation}
The inequality $n_{+}^{P}(x)-n_{+}^{Q}(x)\ge2$ implies the existence of $p_{+}-q_{+}$ many
angles $\theta<\theta_{k}<\pi$, two of which are $\theta_{1}$ and $\theta_{2}$, whose sum is at least $\sum_{k=1}^{p_{+}}\theta_{k}-\sum_{j=1}^{q_{+}}\eta_{j}$. We thus see that 
\[
(\theta_{1}-\theta)+(\theta_{2}-\theta)\ge(p_{+}-q_{+}-1)\pi+(r-2)\theta-(p_{+}-q_{+}-2)\pi\ge\pi.
\]
Consequently, $\theta_{2}-\theta>\pi/2$. The claim $\tau_2>\tau$ now follows by noting that in the triangle $\triangle O\tau_2 \tau e^{i\theta}$ in Figure \ref{fig:angledef} the angle at $\tau e^{i\theta}$ (namely $\theta_2-\theta$) is the largest, and hence the side opposite this vertex is the longest. 
\end{proof}
\begin{lem}
\label{lem:tauexistence} Let $n$ and $s$ denote the total number of zeros of $P$ and $Q$ respectively. For each $\theta\in(0,\pi/r)$, there exists
a unique $t=\tau e^{i\theta}$ for which \eqref{eq:sumangles} holds. 
\end{lem}

\begin{proof} 
We first have the inequalities 
\[
p_{+}-q_{+}-1=n_{+}^{P}(\infty)-n_{+}^{Q}(\infty)-1>0
\]
and 
\[
p_{+}-q_{+}-1\ge\frac{n-s}{2}
\]
since the second inequality is equivalent to 
\[
p_{+}-q_{+}-2\ge p_{-}-q_{-}.
\]
Next, we observe that 
\[
\sum_{-p_{-}<k\le p_{+}}\theta_{k}(t)-\sum_{-q_{-}<k\le q_{+}}\eta_{k}(t)-r\theta
\]
approaches $(p_{+}-q_{+})\pi-r\theta$ as $|t|\rightarrow0$, and
$(n-s-r)\theta$ as $|t|\rightarrow\infty$ where 
\[
(n-s-r)\theta<\frac{(n-s-r)\pi}{r}<(p_{+}-q_{+}-1)\pi<(p_{+}-q_{+})\pi-r\theta.
\]
By the intermediate value theorem, there is a $\tau\in(0,\infty)$
so that \eqref{eq:sumangles} holds. The uniqueness of $t=\tau e^{i\theta}$
comes directly from Lemmas \ref{lem:ImR} and \ref{lem:tauupperbound}. 
\end{proof}
Thus for each $\theta\in(0,\pi/r)$, we can define the functions $\tau(\theta)$,
$\theta_{k}(\theta)$, $-p_{-} < k\le p_{+}$, and $\eta_{j}(\theta)$, $-q_{-} < j\le q_{+}$,
according to \eqref{eq:thetadef}, \eqref{eq:etadef}, \eqref{eq:taudef},
and \eqref{eq:sumangles}. To ensure these functions are analytic,
we need to make use of the complex version of the Implicit Function
Theorem. 
\begin{thm}[Theorem 2.1.2,\, p.24 \cite{hormander}]
\label{IFT} Let $f_{j}(w,z),$ $j=1,\ldots,m$, be analytic functions
of $(w,z)=(w_{1},\ldots,w_{m},z_{1},\ldots,z_{n})$ in a neighborhood
of a point $(w^{*},z^{*})$ in $\mathbb{C}^{m}\times\mathbb{C}^{n}$,
and assume that $f_{j}(w^{*},z^{*})=0$, $j=1,\ldots,m$, and that
\[
\det\left(\frac{\partial f_{j}}{\partial w_{k}}\right)_{j,k=1}^{m}\neq0\qquad\textrm{at}\quad(w^{*},z^{*}).
\]
Then the equations $f_{j}(w,z)=0$, $j=1,\ldots,m$ have a uniquely
determined analytic solution $w(z)$ in a neighborhood of $z^{*}$,
such that $w(z^{*})=w^{*}$. 
\end{thm}

We are now ready to state and prove the following lemma. 
\begin{lem}
The functions $\tau(\theta)$, $\theta_{k}(\theta)$, $-p_{-} < k\le p_{+}$,
and $\eta_{j}(\theta)$, $-q_{-} < j\le q_{+}$, are analytic in a neighborhood
of $(0,\pi/r)$. 
\end{lem}

\begin{proof}
Let $\overrightarrow{\theta}$ and $\overrightarrow{\eta}$ be the
$n$-tuples of $\theta_{k}$ , $-p_{-}<k\le p_{+}$, and $s$-tuples
of $\eta_{j}$, $-q_{-}<j\le q_{+}$. We define the functions $f_{k},g_{j}:\mathbb{C}^{n+s+1}\times\mathbb{C}\to\mathbb{C}$
by 
\begin{align*}
f_{k}(\overrightarrow{\theta},\overrightarrow{\eta},\tau,\theta) & =\tau_{k}\frac{\sin\theta_{k}}{\sin(\theta_{k}-\theta)}-\tau,\qquad(-p_{-}<k\le p_{+})\\
g_{j}(\overrightarrow{\theta},\overrightarrow{\eta},\tau,\theta) & =\gamma_{j}\frac{\sin\eta_{j}}{\sin(\eta_{j}-\theta)}-\tau,\qquad(-q_{-}\le j\le q_{+})
\end{align*}
and 
\[
h(\overrightarrow{\theta},\overrightarrow{\eta},\tau,\theta)=\sum_{-p_{-}<k\le p_{+}}\theta_{k}-\sum_{-q_{-}<j\le q_{+}}\eta_{j}-r\theta-(p_{+}-q_{+}-1)\pi.
\]
Note that for each $\theta\in(0,\pi/r)$, there exist $\theta_{k}$,
$\eta_{k}$, and $\tau$ so that 
\begin{align*}
f_{k}(\overrightarrow{\theta},\overrightarrow{\eta}\tau,\theta) & =0,\qquad(-p_{-}<k\le p_{+})\\
g_{j}(\overrightarrow{\theta},\overrightarrow{\eta},\tau,\theta) & =0,\qquad(-q_{-}<j\le q_{+})\\
h(\overrightarrow{\theta},\overrightarrow{\eta},\tau,\theta) & =0,
\end{align*}
and that there exists a neighborhood $\mathcal{W}$ of $(\overrightarrow{\theta},\overrightarrow{\eta}\tau,\theta)\in\mathbb{C}^{n+s+1}\times\mathbb{C}$
where each of these function is analytic. We calculate 
\begin{align*}
\frac{\partial f_{k}}{\partial\theta_{k}} & =\frac{-\tau_{k}\sin\theta}{\sin^{2}(\theta_{k}-\theta)}=:c_{k},\\
\frac{\partial g_{j}}{\partial\eta_{j}} & =\frac{-\gamma_{j}\sin\theta}{\sin^{2}(\eta_{j}-\theta)}=:d_{j},
\end{align*}
and hence the Jacobian matrix at $(\overrightarrow{\theta},\overrightarrow{\eta},\tau,\theta)$
is 
\[
\left[\begin{array}{ccc}
C & 0 & -1\\
0 & D & -1\\
1 & -1 & 0
\end{array}\right]
\]
where $C$ and $D$ are two $n\times n$ and $s\times s$ diagonal
matrices whose diagonal entries are $c_{k}$, $-p_{-}<k\le p_{+}$,
and $d_{j}$, $-q_{-}<j\le q_{+}$. By expanding along the first row,
we find the determinant of this matrix to be 
\[
\pm\prod_{-p_{-}<k\le p_{+}}c_{k}\prod_{-q_{-}<j\le q_{+}}d_{j}\left(\sum_{-p_{-}<k\le p_{+}}\frac{1}{c_{k}}-\sum_{-q_{-}<j\le q_{+}}\frac{1}{d_{j}}\right).
\]
We now show that this expression is nonzero. To this end note that
equation \eqref{eq:taudef} implies 
\begin{equation}
t-\tau_{k}=\tau_{k}\frac{\cos\theta_{k}\sin\theta+i\sin\theta_{k}\sin\theta}{\sin(\theta_{k}-\theta)}=\tau_{k}\frac{\sin\theta}{\sin(\theta_{k}-\theta)}e^{i\theta_{k}}=\frac{\tau\sin\theta}{\sin\theta_{k}}e^{i\theta_{k}}\label{eq:tau-t}
\end{equation}
for $-p_{-}<k\leq p_{+}$, and 
\begin{equation}
t-\gamma_{j}=\gamma_{k}\frac{\sin\theta}{\sin(\eta_{j}-\theta)}e^{i\eta_{j}}=\frac{\tau\sin\theta}{\sin\eta_{j}}e^{i\eta_{j}}.\label{eq:gamma-t}
\end{equation}
for $-q_{-}<j\leq q_{+}$. Together with \eqref{eq:ImRexp} these
identities yield 
\begin{align*}
\frac{\Im R(t)}{\Im t} & =\sum_{-p_{-}<k\le p_{+}}\frac{\sin^{2}(\theta_{k}-\theta)}{\tau_{k}\sin^{2}\theta}-\sum_{-q_{-}<j\le q_{+}}\frac{\sin^{2}(\eta_{j}-\eta)}{\gamma_{j}\sin^{2}\theta}\\
 & =-\frac{1}{\sin\theta}\left(\sum_{-p_{-}<k\le p_{+}}\frac{1}{c_{k}}-\sum_{-q_{-}<j\le q_{+}}\frac{1}{d_{j}}\right).
\end{align*}
The result now follows from condition (\ref{cond3}) in Theorem \ref{thm:maintheorem}. 
\end{proof}
\begin{figure}
\includegraphics[scale=0.4]{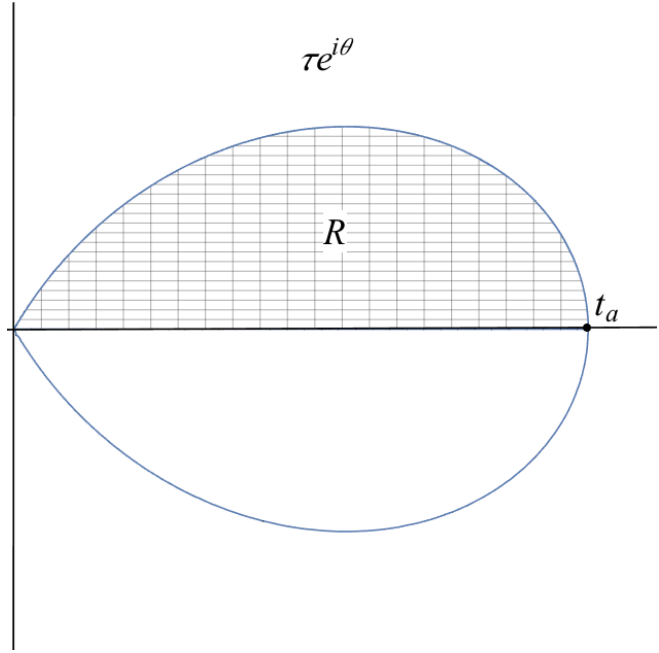} \caption{\label{fig:taucurve}The $\tau(\theta)e^{i\theta}$ curve and the
set $\mathcal{R}$ (c.f. Lemma \ref{lem:nonzeroIm}) for $P(t)=(t-1)(t-2)(t-3)$,
$Q(t)=(t+3)(t-4)$ and $r=3$}
\end{figure}

\begin{defn} \label{def:zthetadef}
We define the function $z(\theta)$
for $\theta\in(0,\pi/r)$ by
\begin{equation}
z(\theta)=-\frac{P(t)}{t^{r}Q(t)},\label{eq:zthetadef}
\end{equation}
where $t=\tau(\theta)e^{i\theta}$.
\end{defn}
\begin{lem}
\label{lem:zmonotone} Under the hypotheses of Theorem \ref{thm:maintheorem}
the function $(-1)^{p_{+}-q_{+}}z(\theta)$ is real valued, increasing
and positive on $(0,\pi/r)$. 
\end{lem}

\begin{proof}
We write \eqref{eq:zthetadef} as 
\[
z(\theta)=-\frac{\prod_{-p_{-}<k\le p_{+}}\left(\tau e^{i\theta}-\tau_{k}\right)}{\tau^{r}e^{ir\theta}\prod_{-q_{-}<j\le q_{+}}(\tau e^{i\theta}-\gamma_{j})}.
\]

With \eqref{eq:tau-t}, \eqref{eq:gamma-t} and \eqref{eq:sumangles},
this equation becomes 
\begin{equation}
z(\theta)=(-1)^{p_{+}-q_{+}}(\tau\sin\theta)^{n-s}\cdot\frac{\prod_{-q_{-}<j\le q_{+}}\sin\eta_{j}}{\prod_{-p_{-}<k\le p_{+}}\sin\theta_{k}}\label{eq:signz}
\end{equation}
from which we deduce that $(-1)^{p_{+}-q_{+}}z(\theta)$ is a positive
real-valued function on $(0,\pi/r)$. We continue by showing that
$\ln|z|$ (and hence $|z|$) is strictly increasing on this interval.
Since $\operatorname{Im}f$ is constant on the $\tau$-curve, by computing
its derivative with respect to $\theta$ there we obtain 
\begin{equation}
0=\frac{\partial\Im f}{\partial\tau}\frac{d\tau}{d\theta}+\frac{\partial\Im f}{\partial\theta}.\label{eq:zerodiff}
\end{equation}
Multiplying both sides of the equation 
\begin{align*}
\frac{d\ln|z|}{d\theta} & =-\frac{\partial\Re f}{\partial\tau}\frac{d\tau}{d\theta}-\dfrac{\partial\Re f}{\partial\theta}
\end{align*}
by $\tau\partial\Im f/\partial\tau$ and combining \eqref{eq:CauchyRiemann}
with \eqref{eq:zerodiff} results in the equation 
\[
\Im R(\tau e^{i\theta})\frac{d\ln|z|}{d\theta}=\left|R(\tau e^{i\theta})\right|^{2}.
\]
By assumption (3) in Theorem \ref{thm:maintheorem} and Lemma \ref{lem:tauupperbound}, $\operatorname{Im}R>0$
on the $\tau$ curve, and hence the right hand side of the above equation
is strictly positive there. The result now follows. 
\end{proof}
\begin{lem}
\label{lem:endlimits}Let $\tau(\theta)$ and $z(\theta)$ be defined
as in (\ref{eq:taudef}) and (\ref{eq:zthetadef}) respectively for
$\theta\in(0,\pi/r)$. Then 
\begin{itemize}
\item[(i)] $\lim_{\theta\rightarrow0}\tau(\theta)=t_{a}$, 
\item[(ii)] $\lim_{\theta\rightarrow0}z(\theta)=-P(t_{a})/t_{a}^{r}Q(t_{a})$, 
\item[(iii)] $\lim_{\theta\rightarrow\pi/r}\tau(\theta)=0$, \quad{}\text{
and} 
\item[(iv)] $\lim_{\theta\rightarrow\pi/r}z(\theta)=(-1)^{p_{+}-q_{+}}\infty$. 
\end{itemize}
\end{lem}

\begin{proof}
Combining the Cauchy-Riemman equations \eqref{eq:CauchyRiemann} with
equation \eqref{eq:zerodiff} we find that along the $\tau$-curve
\[
\dfrac{d\tau}{d\theta}=-\frac{1}{\tau}\frac{\Re R(t)}{\Im R(t)}.
\]
Recall that $R(t)$ is a rational function, and hence the number of
critical points of $\tau(\theta)$ on $(0,\pi/r)$ is finite. Since
$0\leq\tau(\theta)\leq\tau_{2}$, $\tau(\theta)$ is bounded, and
consequently the limits $\lim_{\theta\rightarrow0}\tau(\theta)$ and
$\lim_{\theta\rightarrow\pi/r}\tau(\theta)$ exist. Consider now the
two solutions $t_{0,1}=\tau(\theta)e^{\pm i\theta}$ to the equation
$z(\theta)+\frac{P(t)}{t^{r}Q(t)}=0$. Lemma \ref{lem:zmonotone}
implies that $\lim_{\theta\to0}z(\theta)=:z(0)$ exists, therefore
$\lim_{\theta\to0}\tau(\theta)$ is a double root of $z(0)+P(t)/t^{r}Q(t)$.
As such, it is also a root of 
\[
\dfrac{d}{dt}\left(z(0)+\dfrac{P(t)}{t^{r}Q(t)}\right)=-P(t)R(t)/t^{r+1}Q(t).
\]
Having established these facts, \\
 (i) is now a straightforward consequence of Lemmas \ref{lem:taexistence}
and \ref{lem:tauupperbound};\\
 (ii) follows from (i) and the definition of $z(\theta)$;\\
 (iii) follows from Lemma \ref{lem:ImR} and equation \eqref{eq:sumangles};
and \\
 (iv) is easily seen using (iii) and the definition of $z(\theta)$
to establish that $\lim_{\theta\rightarrow\pi/r}|z(\theta)|=\infty$,
and noting that the sign of $z(\theta)$ is $(-1)^{p_{+}-q_{+}}$
by \eqref{eq:signz}. \\
\end{proof}
In light of Lemma \ref{lem:endlimits}, we will henceforth understand $\tau(\theta)$ to be a continuous function
on $[0,\pi/r]$. The next lemma establishes that the function $P(t)/t^{r}Q(t)$ is real valued on the boundary of the set shown in Figure \ref{fig:taucurve}, but
nowhere in its interior.
\begin{lem}
\label{lem:nonzeroIm}If 
\begin{equation}
\mathcal{R}=\left\{ t=|t|e^{i\theta}\in\mathbb{C}\ |\ 0<\theta<\pi/r\text{ and }0<|t|\le\tau(\theta)\right\} \label{eq:regionR}
\end{equation}
and $t=|t|e^{i\theta}\in\mathcal{R}$, then 
\[
\Im\frac{P(t)}{t^{r}Q(t)}=0
\]
if and only if $t=\tau(\theta)e^{i\theta}$. 
\end{lem}

\begin{proof}
If $t=\tau(\theta)e^{i\theta}$ then by the reality of $z(\theta)$
(c.f. Lemma \ref{lem:zmonotone}) 
\[
\Im\frac{P(t)}{t^{r}Q(t)}=\Im\left(-z(\theta)\right)=0.
\]
If $t=|t|e^{i\theta}\in\mathcal{R}$ and 
\[
\Im\frac{P(t)}{t^{r}Q(t)}=0
\]
then 
\[
(\dag)\qquad\frac{P(t)Q(\overline{t})}{P(\overline{t})Q(t)}=e^{2ri\theta}.
\]
For $-p_{-}<k\le p_{+}$ and $-q_{-}<j\le q_{+}$ we define the angles
$\theta_{k}(t)$ and $\eta_{j}(t)$ via the equations 
\begin{align*}
\frac{\tau_{k}-|t|e^{i\theta}}{\tau_{k}-|t|e^{-i\theta}} & =e^{2i\theta_{k}(t)}\\
\frac{\gamma_{j}-|t|e^{i\theta}}{\gamma_{j}-|t|e^{-i\theta}} & =e^{2i\eta_{j}(t)}.
\end{align*}
Substituting these expressions in $(\dag)$ and equating exponents
yields 
\[
\sum_{-p_{-}<k\le p_{+}}\theta_{k}(t)-\sum_{q_{-}<j\le q_{+}}\eta_{j}(t)=r\theta+l\pi,\qquad\text{for some}\quad l\in\mathbb{Z}.
\]
By Lemma \ref{lem:ImR} (2), for each $\theta$ we have 
\begin{align*}
(p_{+}-q_{+})\pi & >\sum_{-p_{-}<k\le p_{+}}\theta_{k}(t)-\sum_{-q_{-}<k\le q_{+}}\eta_{k}(t)-r\theta\\
 & \ge\sum_{-p_{-}<k\le p_{+}}\theta_{k}(\tau e^{i\theta})-\sum_{-q_{-}<k\le q_{+}}\eta_{k}(\tau e^{i\theta})-r\theta\\
 & =(p_{+}-q_{+}-1)\pi.
\end{align*}
Since $l\in\mathbb{Z},$ we must have $l=p_{+}-q_{+}-1$. By Lemma \ref{lem:tauexistence} we conclude that $|t|=\tau$, and the result follows.
\end{proof}
\begin{lem}
\label{lem:zerosoutsidedisk}Let $\theta\in[0,\pi/r)$ be a fixed
angle with $z:=z(\theta)$ and $\tau:=\tau(\theta)$. The only zeros
in $t$ of $P(t)+zt^{r}Q(t)$ in the closed disk centered at the origin
with radius $\tau$ are $t_{0,1}=\tau e^{\pm i\theta}$. 
\end{lem}

\begin{proof}
Lemma \ref{lem:nonzeroIm} implies that $P(t)+zt^{r}Q(t)$ has no
zero on the region $\mathcal{R}$ in \eqref{eq:regionR} except $\tau e^{\pm i\theta}$.
Suppose $t\notin\mathcal{R}$, $|t|\le\tau$, and $\Im t\ge0$. We
consider four cases.

Suppose $t\in(0,t_{a}]$. Recall that $t_{a}$ as the smallest positive
zero of $P(t)R(t)$. Since $\frac{d}{dt}(P(t)/t^{r}Q(t))=-\dfrac{P(t)R(t)}{t^{r+1}Q(t)}$,
we see that $P(t)/t^{r}Q(t)$ is monotone on $(0,t_{a}]$. The sign
of the derivative is $(-1)^{p_{+}-q_{+}}$ for $t\ll1$  and hence 
\[
(-1)^{p_{+}-q_{+}}\frac{P(t)}{t^{r}Q(t)}\ge(-1)^{p_{+}-q_{+}}\frac{P(t_{a})}{t_{a}^{r}Q(t_{a})}.
\]
On the other hand, Lemmas \ref{lem:zmonotone} and \ref{lem:endlimits}
imply 
\[
(-1)^{p_{+}-q_{+}}z(\theta)\ge(-1)^{p_{+}-q_{+}+1}\frac{P(t_{a})}{t_{a}^{r}Q(t_{a})}.
\]
Combining these inequalities gives 
\[
(-1)^{p_{+}-q_{+}+1}\frac{P(t_{a})}{t_{a}^{r}Q(t_{a})}\geq(-1)^{p_{+}-q_{+}+1}\frac{P(t)}{t^{r}Q(t)}=(-1)^{p_{+}-q_{+}}z(\theta)\geq(-1)^{p_{+}-q_{+}+1}\frac{P(t_{a})}{t_{a}^{r}Q(t_{a})}.
\]
The conclusion $t=t_{a}$ now follows. \\
 Suppose next that $0\le\Arg t\le\theta$ and $t\notin(0,t_{a}]$.
Then $|t|>\tau(\Arg t)$. \textcolor{black}{By the intermediate value
theorem, there is $\theta^{*}\in(\Arg t,\theta]$ such that $\tau(\theta^{*})=|t|$}\textcolor{red}{.}
Lemmas \ref{lem:ImR} and \ref{lem:zmonotone} imply that 
\[
\left|\frac{P(t)}{t^{r}Q(t)}\right|<(-1)^{p_{+}-q_{+}} z(\theta^{*})\le(-1)^{p_{+}-q_{+}}z(\theta).
\]

Similarly, if $\theta<\Arg t<\pi/r$, then the inequality $|t|>\tau(\Arg t)$
implies that \textcolor{black}{there is $\theta^{*}\in(\theta,\Arg t)$
such that $\tau(\theta^{*})=|t|$} and thus 
\[
\left|\frac{P(t)}{t^{r}Q(t)}\right|>(-1)^{p_{+}-q_{+}}z(\theta^{*})>(-1)^{p_{+}-q_{+}}z(\theta).
\]

Finally, if $\pi/r\le\Arg t\le\pi$, then Lemma \ref{lem:endlimits}
implies that there is $\theta^{*}\in(\theta,\pi/r)$ such that $\tau(\theta^{*})=|t|$
and thus 
\[
\left|\frac{P(t)}{t^{r}Q(t)}\right|>(-1)^{p_{+}-q_{+}}z(\theta^{*})>(-1)^{p_{+}-q_{+}}z(\theta).
\]
In all cases we showed that if $t$ is in the closed disk with radius $\tau$ and $t \neq \tau e^{\pm i\theta}$, then $t$ cannot be a zero of $P(t)/t^r Q(t)$. The proof is thus complete.
\end{proof}

\section{Zeros of $H_{m}(z)$} \label{sec:completetheproof}

Let $t_{k}$, $0\le k<\max(n,r+s)$, be the zeros of $P(t)+zt^{r}Q(t)$.
If these zeros are distinct, then the Cauchy integral formula gives
\begin{align}
H_{m}(z) & =\frac{1}{2\pi i}\oint_{|t|=\epsilon}\frac{dt}{(P(t)+zt^{r}Q(t))t^{m+1}}\nonumber \\
 & =-\sum_{0\le k<\max(n,r+s)}\frac{1}{D_{t}(t_{k},z)t_{k}^{m+1}},\label{eq:cauchyint}
\end{align}
where
\begin{align*}
D_{t}(t_{k},z) & =P'(t_{k})-\frac{P(t_{k})}{t_{k}^{r}Q(t_{k})}\left(rt_{k}^{r-1}Q(t_{k})+t_{k}^{r}Q'(t_{k})\right)\\
 & =-\frac{P(t_{k})R(t_{k})}{t_{k}}.
\end{align*}
Substituting this expression into  (\ref{eq:cauchyint}) we obtain 
\begin{equation}
H_{m}(z)=\sum_{0\le k<\max(n,r+s)}\frac{1}{P(t_{k})R(t_{k})t_{k}^{m}}.\label{eq:Hmexpsum}
\end{equation}
We will have to exercise care when using this representation. In particular, we now know that $t_0$ and $t_1$ are distinct zeros of $P(t)+zt^{r}Q(t)$ (c.f. Lemma \ref{lem:zerosoutsidedisk}), but without further information on the remaining zeros can only separate the residues of the generating function corresponding to these, and will have to leave the remainder in an integral representation.\\ 
Let $g(t)=P(t)R(t)t^{m}$ and denote by $\left\{ \theta_{h}\right\} $,
$0\le\theta_{h}\le\pi/r$, the sequence of angles which correspond to the points $s_{h}:=\tau(\theta_{h})e^{i\theta_{h}}$ on the $\tau$-curve where 
\[
\Im g\left(s_{h}\right)=0.
\]
Let denote $\sigma\left(H_{m}(z(\theta_{h}); 0,\pi/r\right)$ the number
of sign changes of the sequence $\left\{ H_{m}(z(\theta_{h})\right\} _{h}$
on the interval $[0,\pi/r]$ where $z(0)$ and $z(\pi/r)$ are defined
by the limits in Lemma \ref{lem:endlimits}. We now show that 
\[
\sigma\left(H_{m}(z(\theta_{h}); 0,\pi/r\right)\geq \left\lfloor m/r\right\rfloor,
\]
and hence $H_{m}(z)$ has at least $\left\lfloor m/r\right\rfloor $
real zeros between ${\displaystyle -\frac{P(t_{a})}{t_{a}^{r}Q(t_{a})}}$
and $(-1)^{p_{+}-q_{+}}\infty$ by the intermediate value theorem
and Lemmas \ref{lem:zmonotone} and \ref{lem:endlimits}.

In order to establish a lower bound on the quantity $\sigma\left(H_{m}(z(\theta_{h}); 0,\pi/r\right)$, we first study the sign changes in the sequence $\left \{g(s_{h})\right \}$. To this end, let $\gamma$ be the counterclockwise
loop formed by the curve $\tau(\theta)e^{i\theta}$ and its conjugate
(see Figure \ref{fig:taucurve}). Denote by $\gamma'$ the image of $\gamma$
under the map $g(t)-\xi$, where $\xi\ne0$ is a small real number chosen so that $\xi P(0)R(0)>0$.
According to the Argument Principle, the winding number of $\gamma'$
around the origin is equal to the number of zeros of $g(t)-\xi$ inside $\gamma$,
since this function has no poles there. 
If $g(t)-\xi=0$, truncating the Taylor expansion of $P(t)R(t)$ about the origin at the constant term yields
\[
P(0)R(0)t^{m}\left(1+\mathcal{O}(t)\right)=\xi.
\]
Rearranging for $t$ yields
\[
t=\omega_{k}\sqrt[m]{\frac{\xi}{P(0)R(0)}}\left(1+\mathcal{O}(\xi{}^{1/m})\right),
\]
where $\omega_{k}=e^{2k\pi i/m}$ is an $m$-th root of $1$. \\
If we truncate the Taylor expansion of $P(t)R(t)$ about the origin at the linear term instead, we obtain
\[
P(0)R(0)t^{m}\left(1+\left(\frac{P'(0)}{P(0)}+\frac{R'(0)}{R(0)}\right)t+\mathcal{O}(t^{2})\right)=\xi,
\]
which -- after using our first estimate for $t$ -- leads to the more precise estimate 
\begin{equation} \label{eq:tasympt}
t=\omega_{k}\sqrt[m]{\frac{\xi}{P(0)R(0)}}\left(1-\frac{\omega_{k}}{m}\left(\frac{P'(0)}{P(0)}+\frac{R'(0)}{R(0)}\right)\sqrt[m]{\frac{\xi}{P(0)R(0)}}+\mathcal{O}(\xi^{2/m})\right).
\end{equation}
Computing the principal argument of both sides gives
\begin{equation}\label{eq:zerocountasympt}
\Arg t=\frac{2k\pi}{m}-\frac{\sin(2k\pi/m)}{m}\left(\frac{P'(0)}{P(0)}+\frac{R'(0)}{R(0)}\right)\sqrt[m]{\frac{\xi}{P(0)R(0)}}+\mathcal{O}(\xi^{2/m}).
\end{equation}
Equations (\ref{eq:tasympt}) and (\ref{eq:zerocountasympt}) establish that as $\xi \to 0$, $|t| \to 0$, while $\Arg(t) \to 2 k \pi /m$. Thus,  for sufficiently small $\xi$, $|t|<\tau(\Arg t)$,  and consequently
$t$ lies inside $\gamma$ if $|\Arg(t)|\leq \pi/r-\nu$ for some fixed small $\nu$ independent of $\xi$. Thus $g(t)-\xi$ has at least
\[
2\left\lfloor m/2r\right\rfloor +1
\]
zeros inside $\gamma$ close to the origin if $2r\nmid m$ (namely one for each value of $k$ between $-\lfloor m/2r\rfloor$ and $\lfloor m/2r\rfloor$), while it has at least $m/r-1$ such zeros in case $2r\mid m$.\\
In addition to the zeros found by the above asymptotic expansion, $g(t)-\xi$ has an additional zero near $t_{a}$. Indeed, $g(0)=g(t_a)=0$, $g(t) \neq 0$ on $(0,t_{a})$ and $g(t)$ is continuous and real valued on $[0,t_a]$. Furthermore, from the argument above, $g(t)-\xi$
has a simple real zero close to the origin for sufficiently small
$\xi$. It follows that $g(t)-\xi$ must also have a real zero near (but to the left of) $t_{a}$, which therefore lies inside $\gamma$. Applying argument principle to $g(t)-\xi$ we conclude that 
\[
\frac{1}{2\pi i}\oint_{\gamma}\frac{g'(t)}{g(t)-\xi}dt\ge\begin{cases}
2\left\lfloor m/2r\right\rfloor +2 & \text{ if }2r\nmid m\\
m/r & \text{ if }2r|m
\end{cases}.
\]
Note that we may partition the curve $\gamma$ into arcs $[s_{h},s_{h+1}]\subset\gamma$
and their conjugates. Having done so we obtain 
\[
\frac{1}{2\pi i}\oint_{\gamma}\frac{g'(t)}{g(t)-\xi}dt=\frac{1}{2\pi i}\sum_{h}\int_{[s_{h},s_{h+1}]\cup[\overline{s_{h+1}},\overline{s_{h}}]}\frac{g'(t)}{g(t)-\xi}dt.
\]
By the definition of the points $\{s_{h}\}$, $g(t)-\xi$ maps $[s_{h},s_{h+1}]\cup[\overline{s_{h+1}},\overline{s_{h}}]$ to a loop whose winding number around $0$ is nonzero if and only
if 
\[
\left(g(s_{h+1})-\xi\right)\left(g(s_{h})-\xi\right)<0,
\]
which -- by taking $\xi$ is sufficiently small -- implies that $g(s_{h+1})g(s_{h})\le0$.
Since $g(s_{h})=0$ if and only $\theta_{h}=0$ or $\theta_{h}=\pi/r$, we conclude that 
\[
\sigma\left(g(s_{h}),0,\pi/r\right)\ge\begin{cases}
2\left\lfloor m/2r\right\rfloor  & \text{ if }2r\nmid m\\
m/r-2 & \text{ if }2r|m
\end{cases}.
\]
In the subsequent section we prove that
\begin{itemize}
\item[(i)] For $m \gg 1$ the sign of $g(s_{h})$ is the same as the sign of $H_{m}(z(\theta_{h}))$; 
\item[(ii)] if $2r|m$ and one endpoint of $(\theta_{h},\theta_{h+1})$ is $\pi/r$
then $H_{m}(z(\theta))$ has a zero on this interval, and
\item[(iii)] the sign of $H_{m}(z(\theta))$ is $(-1)^{p_{+}}$ and $(-1)^{p_{+}+\left\lfloor m/r\right\rfloor }$
as $\theta\rightarrow0$ and $\theta\rightarrow(\pi/r)^-$ respectively. 
\end{itemize}
Using these three results we now complete the proof of Theorem \ref{thm:maintheorem}. By the intermediate value theorem $H_{m}(z(\theta))$ has at least $\left\lfloor m/r\right\rfloor -1$ zeros on $(0,\pi/r)$ each of which gives a distinct zero of $H_{m}(z)$
between $a$ and $(-1)^{p_{+}-q_{+}}\infty$ by the monotonicity of
$z(\theta)$. Since the degree of $H_{m}(z)$ is $\left\lfloor m/r\right\rfloor $ (see (iii) in Remark \ref{rem:negzeroP}), its remaining zero must be real, and the sign of $H_{m}(z)$ as $z\rightarrow(-1)^{p_{+}-q_{+}+1}\infty$ is $(-1)^{\left\lfloor m/r\right\rfloor }$ times the sign of $H_{m}(z)$ as $z\rightarrow(-1)^{p_{+}-q_{+}}\infty$. Thus by (iii), $H_{m}(z)$
has the same sign at $z=a$ as it does near $(-1)^{p_{+}-q_{+}+1}\infty$.
We conclude that $H_m(z)$ cannot change sign between $a$
and $(-1)^{p_{+}-q_{+}+1}\infty$, and must therefore have its remaining zero between $a$ and  $(-1)^{p_{+}-q_{+}}\infty$.
\section{Proofs of the three lemmas} \label{sec:threelemmas}
We begin this section with a result concerning exponential polynomials. While we will use it to establish the three claims made at the end of the previous section, the result is interesting in its own right, as exponential polynomials are objects of interest in a number of active research areas. Without striving for a completeness, we mention only a few here. Shapiro's 1958 conjecture on the zero distribution of the members of the ring or exponential polynomials motivated D'Aquino, Macintyre and Terzo to study these objects in an algebraic setting in the paper \cite{dmt}), where they attribute the origins of Shapiro's conjecture to complex analytic considerations. Exponential polynomials are also central objects in the study of decomposition of integers into sums of powers of integers (such as Vinogradov's Three primes theorem) and in the methods used in arriving at such theorems (such as the Hardy-Littlewood circle method). Finally, we note that they also appear in Weyl's criterion regarding the equidistribution of sequences. We are unaware of results akin to that of Proposition \ref{prop:signexppoly}, which shows that certain exponential polynomials have infinitely many real zeros.
\begin{prop}
\label{prop:signexppoly}Let $n \in \mathbb{N}$, and let $\omega_{k}=e^{(2k-1)\pi i/n}$ be the
$n$-th root of $-1$. For any $\ell\in\mathbb{Z}$, and $x\ge0$ such
that $\omega_{1}^{\ell}e^{x\omega_{1}}=\omega_{0}^{\ell}e^{x\omega_{0}}\in\mathbb{R}$,
the sign of 
\begin{equation}
\sum_{k=0}^{n-1}\omega_{k}^{\ell}e^{x\omega_{k}}\label{eq:exppoly}
\end{equation}
is the same as the sign of the first term. In particular, the function \eqref{eq:exppoly} -- as a function of $x$ -- has infinitely many real zeros.
\end{prop}
\begin{proof} The result is immediate for the cases $n=1, 2$. We henceforth assume that $n \geq 3$, and (without loss of generality) that $0\leq \ell < n$. The requirements that $\omega_{1}^{\ell}e^{x\omega_{1}}=\omega_{0}^{\ell}e^{x\omega_{0}}\in\mathbb{R}$ and $x\geq0$ necessitate that 
\[
x=\frac{\pi\left(b-\frac{\ell}{n}\right)}{\sin\left(\frac{\pi}{n}\right)},\qquad \textrm{for some} \ b \in\mathbb{Z}^{+}.
\]
With this explicit expression for $x$ we also calculate 
\begin{equation}\label{eq:signdominanceexpoexplain}
\omega_{0}^{\ell}e^{x\omega_{0}}=(-1)^{b}e^{\pi\left(b-\frac{\ell}{n}\right)\cot\left(\pi/n\right)}.
\end{equation}
Thus, if $x$ is as required, then 
\[
\sum_{k=0}^{n-1}\omega_{k}^{\ell}e^{x\omega_{k}}=\omega_{0}^{\ell}e^{x\omega_{0}}\left(2+\sum_{k=2}^{n-1}\left(\frac{\omega_{k}}{\omega_{0}}\right)^{\ell}e^{x(\omega_{k}-\omega_{0})}\right)
\]
and consequently, it remains to prove that 
\[
2+\sum_{k=2}^{n-1}\frac{\omega_{k}^{\ell}e^{x\omega_{k}}}{\omega_{0}^{\ell}e^{x\omega_{0}}}>0.
\]
We note that 
\begin{eqnarray*}
\left| \sum_{k=2}^{n-1}\frac{\omega_{k}^{\ell}e^{x\omega_{k}}}{\omega_{0}^{\ell}e^{x\omega_{0}}} \right| &\leq&\sum_{k=2}^{n-1} \left| \frac{\omega_{k}^{\ell}e^{x\omega_{k}}}{\omega_{0}^{\ell}e^{x\omega_{0}}} \right| \\
&=&\sum_{k=2}^{n-1}\exp^{-1}\left(x\cos\frac{\pi}{n}-x\cos\frac{(2k-1)\pi}{n}\right)\\ 
& = & \sum_{k=2}^{n-1}\exp^{-1}\left(2x\sin\frac{k\pi}{n}\sin\frac{(k-1)\pi}{n}\right)\\
 & \stackrel{(\star)}{<} & 2\sum_{k=2}^{\left\lfloor n/2\right\rfloor }\exp^{-1}\left(\frac{x(k-1)^{2}\pi^{2}}{2n^{2}}\right)\\
 & < & 2\exp^{-1}\left(\frac{x\pi^{2}}{2n^{2}}\right)+2\int_{1}^{\infty}\exp^{-1}\left(\frac{x\pi^{2}}{2n^{2}}t^{2}\right)dt,
\end{eqnarray*}
where inequality $(\star)$ follows from the fact that $2\sin(t)>t$, for all $t<\pi/2$. Observe that 
\begin{align*}
\int_{1}^{\infty}\exp^{-1}\left(\frac{x\pi^{2}}{2n^{2}}t^{2}\right)dt & =\frac{\sqrt{2}n}{\pi\sqrt{x}}\int_{\pi\sqrt{x}/\sqrt{2}n}^{\infty}e^{-t^{2}}dt\\
 & \le\frac{2n^{2}}{\pi^{2}x}\int_{\pi\sqrt{x}/\sqrt{2}n}^{\infty}te^{-t^{2}}dt\\
 & =\frac{n^{2}}{\pi^{2}x}e^{-\pi^{2}x/2n^{2}}.
\end{align*}
We thus deduce that if $x>n^{2}/8$, then
\[
\sum_{k=2}^{n-1}\exp^{-1}\left(x\cos\frac{\pi}{n}-x\cos\frac{(2k-1)\pi}{n}\right)<2.
\]
We next consider the case when $x\le n^{2}/8$. It is easy to verify that
\begin{eqnarray*}
\frac{\pi}{n}\ge\sin\frac{\pi}{n}&\ge& \frac{\pi}{n}\left(1-\frac{\pi^{2}}{6n^{2}}\right), \qquad (n \ge 1) \\
1+2t&>&\frac{1}{1-t} \hspace{0.8 in}( t\in(0,1/2)), 
\end{eqnarray*}
and consequently
\begin{equation} \label{eq:ineqest}
nb-\ell\le x=\frac{b-\ell/n}{\sin(\pi/n)}\pi\le\left(nb-\ell\right)\left(1+\frac{\pi^{2}}{3n^{2}}\right).
\end{equation}
The first inequality in (\ref{eq:ineqest}) implies that 
\[
nb-\ell\le\frac{n^{2}}{8}
\]
which, when put in the second inequality in (\ref{eq:ineqest}) yields
\[
x\le nb-\ell+\frac{\pi^{2}}{24}.
\]
We write $x=nb-\ell+\delta$ where $0\le\delta\le\pi^{2}/24$, and expand $e^{x \omega_k}$ in a Maclaurin
series to obtain
\begin{align*}
\sum_{k=0}^{n-1} \omega_{k}^{\ell}e^{x\omega_{k}} & =\sum_{k=0}^{n-1}\omega_{k}^{\ell}\sum_{j=0}^{\infty}\frac{(x\omega_{k})^{j}}{j!}\\
 & =n\sum_{j=1}^{\infty}(-1)^{j}\frac{x^{jn-l}}{(jn-l)!}=:n\sum_{j=1}^{\infty}a_{j}.
\end{align*}
We note that the sequence $|a_{j}|$ is increasing when $j\le b-1$
and increasing when $j\ge b+1$. Since the series is alternating, the inequalities
\[
\left|\sum_{j\le b-1}a_{j}\right|\le|a_{b-1}|\qquad\text{and}\qquad\left|\sum_{j\ge b+1}a_{j}\right|\le|a_{b+1}|
\]
are immediate. Thus, the sign of \eqref{eq:exppoly} is $(-1)^{b}$, provided that
\begin{equation}
|a_{b}|>\left|a_{b-1}\right|+\left|a_{b+1}\right| \label{eq:altseriesest}
\end{equation}
with the convention that $a_{0}=0$. In order to establish (\ref{eq:altseriesest}), we
observe that if $b>1$, then 
\begin{align*}
\ln\frac{|a_{b-1}|}{|a_{b}|} & =\ln\frac{\prod_{j=1}^{n}\left(n(b-1)-\ell+j\right)}{x^{n}}\\
 & =\ln\prod_{j=1}^{n}\left(1-\frac{n-j+\delta}{x}\right)\\
 & <-\sum_{j=1}^{n}\frac{n-j+\delta}{x}\\
 & =-\frac{n\delta}{x}-\frac{n^{2}}{2x}\left(1-\frac{1}{n}\right)\\
 & <-4\left(1-\frac{1}{n}\right).
\end{align*}

Similarly
\begin{align*}
\ln\frac{\left|a_{b+1}\right|}{\left|a_{b}\right|} & =\ln\frac{x^{n}}{\prod_{j=1}^{n}(nb-\ell+j)}\\
 & =\ln\prod_{j=1}^{n}\left(1+\frac{j-\delta}{x}\right)^{-1}.
\end{align*}
The inequalities
\begin{eqnarray*}
t/2 &<& \ln(1+t)\qquad \textrm{for all} \quad t\in(0,2), \quad \textrm{and}\\
0<j-\delta&\le& n-\delta<2x
\end{eqnarray*}
imply that  
\begin{align*}
\ln\prod_{j=1}^{n}\left(1+\frac{j-\delta}{x}\right)^{-1}& \leq-\sum_{j=1}^{n}\frac{j-\delta}{2x} \\
 & =-\frac{n(n+1)}{4x}+\frac{n\delta}{2x}\\
 & <-\frac{n^{2}}{2x}\left(\frac{1}{2}-\frac{\delta}{n}\right)\\
 & <-4\left(\frac{1}{2}-\frac{\pi^{2}}{24n}\right).
\end{align*}
Thus 
\[
\frac{|a_{b-1}|}{|a_{b}|}+\frac{|a_{b+1}|}{|a_{b}|}<\exp\left(-4+\frac{4}{n}\right)+\exp\left(-2+\frac{\pi^{2}}{6n}\right)<1
\]
since $n\ge 3$. This establishes that the sign of \eqref{eq:exppoly} is determined by that of its first term. \\
With this result in hand it is now easy to see that \eqref{eq:exppoly} has infinitely many zeros, since the right hand side of \eqref{eq:signdominanceexpoexplain} changes signs infinitely many times as $x$ ranges through the real numbers. The proof is thus complete.
\end{proof}
\subsection*{The three lemmas}
We now prove the three statements preceding the completion of the proof of Theorem \ref{thm:maintheorem}. We begin by showing that the function $g(s_h)$ and $H_m(z(\theta_h))$ have the same for appropriately chosen values of $\theta_h$.
\begin{lem}\label{lem:lemma1of3} Let $g(t)=P(t)R(t)t^{m}$ and denote by $\left\{ \theta_{h}\right\} $, $0\le\theta_{h}\le\pi/r$, the sequence of angles corresponding to the points $s_{h}:=\tau(\theta_{h})e^{i\theta_{h}}$ on the $\tau$-curve where 
\[
\Im g\left(s_{h}\right)=0.
\]
For all $m\gg 1$, the sign of $g(s_{h})$ is the same as the sign of $H_{m}(z(\theta_{h}))$ for all values of $h$ under consideration.
\end{lem}
\begin{proof} The proof is by cases. \\
\textbf{Case 1:} $\gamma\le\theta_{h}\le\pi/r-\gamma$ for some small fixed $\gamma$ (independent of $m$)

Lemma \ref{lem:zerosoutsidedisk} implies there exists $ \epsilon >0$ such that if $2\le k<\max(n,r+s)$, then $|t_{k}|>\tau(1+2\epsilon)$. Consequently, for angles satisfying $\gamma \le \theta \le \pi/r-\gamma$,
\[
\tau^{m}(\theta)H_{m}(z(\theta))=2\Re\frac{\tau^{m}(\theta)}{P(t_{1}(\theta))R(t_{1}(\theta))t_{1}(\theta)^{m}}+\frac{\tau^{m}(\theta)}{2\pi i}\oint_{|t|=\tau(\theta)(1+\epsilon)}\frac{dt}{(P(t)+z(\theta)t^{r}Q(t))t^{m+1}}. 
\]
Note that on the contour of integration $P(t)+zt^{r}Q(t)$ is bounded away from $0$, hence the integral approaches $0$ as $m\rightarrow\infty$. The sign of $H_{m}(z)$ therefore is the same as the sign of $2\Re\frac{\tau^{m}}{P(t_{1})R(t_{1})t_{1}^{m}}$ provided this expression does not also approach $0$. Since $t_1(\theta)=\tau(\theta)e^{i \theta}$, we see that $t_{1}(\theta_h)=s_{h}$ for all $h$ under consideration.  At these points $g(t_{1}(\theta_h))\in\mathbb{R}$, and consequently the modulus of the first term is 
\[
\frac{2}{|P(t_{1}(\theta_h))R(t_{1}(\theta_h))|},
\]
which is bounded away from zero on the compact set $\gamma\le\theta_{h}\le\pi/r-\gamma$. It follows that the sign of $H_{m}(z(\theta_{h}))$ is the same as the sign of $g(s_{h})$ when $\gamma \le\theta_{h}\le \pi/r - \gamma$. 

\subsection*{Case 2: $\theta_{h}\rightarrow0$ as $m\rightarrow\infty$}

Let $\rho$ be the multiplicity of the zero $\tau_{1}$ of $P(t)$. Suppose first that $\rho=1$, that is, $\tau_1 <t_a < \tau_2$. As $\theta\rightarrow0$, the polynomial $P(t)+zt^{r}Q(t)$
approaches 
\[
P(t)+at^{r}Q(t)=\left(\frac{P(t)}{t^r Q(t)}+a\right)t^r Q(t),
\] 
which has a real zero at $t=t_{a}$ with multiplicity at least two, as a complex conjugate pair of zeros converge there.
This means in particular that
\begin{equation}\label{eq:fprimeatt_a}
\frac{d}{dt} \left(\frac{P(t)}{t^r Q(t)} \right) \Big|_{t=t_a}=0.
\end{equation}
On the other hand,  
\[
R(t)=\frac{\displaystyle{t_a \frac{d}{dt} \left(\frac{P(t)}{t^r Q(t)} \right)}}{\displaystyle{\frac{P(t)}{t^r Q(t)}}},
\]
and hence using equations \eqref{eq:fprimeatt_a} and \eqref{eq:Rprime} we conclude that
\[
0 \stackrel{\eqref{eq:Rprime}}{\neq} R'(t_a)\stackrel{\eqref{eq:fprimeatt_a}}{=}\frac{\displaystyle{t \frac{d^2}{dt^2} \left(\frac{P(t_a)}{t_a^r Q(t_a)} \right)}}{\displaystyle{\frac{P(t_a)}{t_a^r Q(t_a)}}}.
\]
Consequently, $\frac{d^2}{dt^2} \left(\frac{P(t_a)}{t_a^r Q(t_a)} \right) \neq 0$, and we conclude that the multiplicity of $t_a$ as a zero of $P(t)+zt^{r}Q(t)$ is exactly two. Thus the remaining zeros $t_k$, $2 \leq k \leq \max\{n,r\}$ still satisfy $|t_k|>(1+\epsilon)$ for some $\epsilon >0$, and the argument we gave in Case 1 still applies.\\
We next consider the case $\rho>1$. If we define the angles $\theta_{j}^{*}$,
$-p_{-}<j\le p_{+}$, and $\eta_{j}^{*}$, $-q_{-}<j\le q_{+}$, by
\[
\theta_{j}=\begin{cases}
-\theta_{j}^{*} & \text{ if }-p_{-}<j\le0\\
\pi-\pi/\rho-\theta_{j}^{*} & \text{ if }1\le j\le\rho\\
\pi-\theta_{j}^{*} & \text{ if }\rho<j\le p_{+}
\end{cases}\qquad\text{and}\qquad\eta_{j}=\begin{cases}-\eta_{j}^{*} & \text{if }-q_{-}<j\le0\\
\pi-\eta_{j}^{*} & \text{if }1\le j\le q_{+}
\end{cases},
\]
then equation \eqref{eq:sumangles} and Lemmas \ref{lem:taexistence}
and \ref{lem:endlimits} imply that $\theta_{j}^{*}\rightarrow0$ and $\eta_{j}^{*}\rightarrow0$
as $\theta\rightarrow0$. For $\theta_1$ we obtain the following estimate: 
\begin{eqnarray*}
\frac{\sin\theta_{1}}{\sin(\theta_{1}-\theta)} & = & \frac{\sin(\pi/\rho+\theta_{1}^{*})}{\sin(\pi/\rho+\theta_{1}^{*}+\theta)}\\
 & = & \frac{\sin(\pi/\rho)+\cos(\pi/\rho)\theta_{1}^{*}+\mathcal{O}(\theta_{1}^{*2})}{\sin(\pi/\rho)+\cos(\pi/\rho)(\theta_{1}^{*}+\theta)+\mathcal{O}(\theta_{1}^{*2}+\theta_{1}^{*}\theta+\theta^{2})}\\
 & = & 1-\left(\cot\frac{\pi}{\rho}\right)\theta+\mathcal{O}(\theta_{1}^{*2}+\theta_{1}^{*}\theta+\theta^{2})
\end{eqnarray*}
The corresponding estimate for the cases $-p_{-}<j\le0$ or $\rho<j$ is given by
\[
\frac{\sin\theta_{j}}{\sin(\theta_{j}-\theta)}=\frac{\theta_{j}^{*}+\mathcal{O}(\theta_{j}^{*3})}{\theta_{j}^{*}+\theta+\mathcal{O}((\theta_{j}^{*}+\theta)^{3})}=\frac{\theta_{j}^{*}}{\theta_{j}^{*}+\theta}\left(1+\mathcal{O}(\theta_{j}^{*2}+\theta^{2}+\theta_{j}^{*}\theta)\right).
\]
Similarly if $-q_{-}<j\le q_{+}$, then 
\[
\frac{\sin\eta_{j}}{\sin(\eta_{j}-\theta)}=\frac{\eta_{j}^{*}}{\eta_{j}^{*}+\theta}\left(1+\mathcal{O}(\eta_{j}^{*2}+\theta^{2}+\eta_{j}^{*}\theta)\right).
\]
For indices satisfying $-p_{-}<j\le0$ or $\rho<j$, the identity 
\[
\frac{\tau_{1}\sin\theta_{1}}{\sin(\theta_{1}-\theta)}=\frac{\tau_{j}\sin\theta_{j}}{\sin(\theta_{j}-\theta)}
\]
gives 
\[
(\theta_{j}^{*}+\theta)\tau_{1}-\tau_{1}\left(\cot\frac{\pi}{\rho}\right)\theta(\theta_{j}^{*}+\theta)=\tau_{j}\theta_{j}^{*}+\mathcal{O}((\theta_{1}^{*}+\theta_{j}^{*}+\theta)^{3})
\]
which we solve for $\theta_{j}^{*}$ and obtain
\begin{eqnarray*}
\theta_{j}^{*} & = & \frac{\tau_{1}\left(\theta-\cot(\pi/\rho)\theta^{2}\right)}{\tau_{j}-\tau_{1}+\tau_{1}\cot(\pi/\rho)\theta}+\mathcal{O}((\theta_{1}^{*}+\theta)^{3})\\
 & = & \left(\frac{\tau_{1}}{\tau_{j}-\tau_{1}}\theta-\tau_{1}\frac{\cot(\pi/\rho)\theta^{2}}{\tau_{j}-\tau_{1}}\right)\left(1-\frac{\tau_{1}\cot(\pi/\rho)\theta}{\tau_{j}-\tau_{1}}\right)+\mathcal{O}((\theta_{1}^{*}+\theta)^{3})\\
 & = & \frac{\tau_{1}}{\tau_{j}-\tau_{1}}\theta-\frac{\cot(\pi/\rho)\tau_{1}\tau_{j}}{(\tau_{j}-\tau_{1})^{2}}\theta^{2}+\mathcal{O}((\theta_{1}^{*}+\theta)^{3}).
\end{eqnarray*}
With the similar identity 
\[
\eta_{j}^{*}=\frac{\tau_{1}}{\gamma_{j}-\tau_{1}}\theta-\frac{\cot(\pi/\rho)\tau_{1}\gamma_{j}}{(\gamma_{j}-\tau_{1})^{2}}\theta^{2}+\mathcal{O}((\theta_{1}^{*}+\theta)^{3}),\qquad-q_{-}<j\le q_{+},
\]
and 
\[
\sum_{-p_{-}<j\le p_{+}}\theta_{j}^{*}-\sum_{-q_{-}<j\le q_{+}}\eta_{j}^{*}+r\theta=0,
\]
we deduce that 
\begin{align*}
\rho\theta_{1}^{*} & =-\left(\sum_{j\le0\text{ or }j>\rho}\frac{\tau_{1}}{\tau_{j}-\tau_{1}}+r-\sum_{-q_{-}<j\le q_{+}}\frac{\tau_{1}}{\gamma_{j}-\tau_{1}}\right)\theta\\
 & +\left(\sum_{j<0\text{ or }j>\rho}\frac{\cot(\pi/\rho)\tau_{1}\tau_{j}}{(\tau_{j}-\tau_{1})^{2}}-\sum_{-q_{-}<j\le q_{+}}\frac{\cot(\pi/\rho)\tau_{1}\gamma_{j}}{(\gamma_{j}-\tau_{1})^{2}}\right)\theta^{2}+\mathcal{O}(\theta^{3}).
\end{align*}

We now turn our attention to the representation given in \eqref{eq:Hmexpsum}. By Lemma \ref{lem:endlimits},
as $\theta\rightarrow0$, the function $P(t)+zt^{r}Q(t)$ approaches
$P(t)$ since $z\rightarrow a=0$. Thus, by an argument similar to
Case 1, it is sufficient to consider the sign of sum 
\begin{equation}
\sum_{0\le k<\rho}\frac{1}{P(t_{k})R(t_{k})t_{k}^{m}} \quad \textrm{as} \quad t_{k}\rightarrow t_{a}=\tau_{1}, \quad (0\le k<\rho)\label{eq:Hmsmalltheta}
\end{equation}
if after multiplication by $\tau^{m}$, the summation does not approach
$0$ as $m\rightarrow\infty$ (which is the case, as can be seen from equation \eqref{eq:Hmasympsmalltheta} and the expression in \eqref{eq:expsumsmalltheta})
Recall that
\[
z(\theta)=-\frac{P(t)}{t^r Q(t)},
\]
and by the definition of the $t_k$s,
\[
P(t_{k})+zt_{k}^{r}Q(t_{k})=0.
\]
Combining these two equations and rearranging yields the equation
\[
\frac{P(t_{k})}{P(\tau e^{i\theta})}-\frac{Q(t_{k})}{Q(\tau e^{i\theta})}\left(\frac{t_{k}}{\tau e^{i\theta}}\right)^{r}=0, \qquad (1 \leq k < \rho).
\]
We let $t_{k}=\tau_{1}+\tau_{1}\epsilon_{k}$, $0\le k<\rho$, and expand the left hand side in a Taylor series centered at $\tau_{1}$ using the identity (c.f. equation \eqref{eq:tau-t})
\[
\tau e^{i\theta}-\tau_{1}=\tau_{1}\frac{\sin\theta}{\sin(\theta_{1}-\theta)}e^{i\theta_{1}}.
\]
Doing so produces
\[
-\frac{\sin^{\rho}(\pi/\rho)\epsilon_{k}^{\rho}}{(-1)^{\rho}\theta^{\rho}}\left(1+\mathcal{O}\left(\epsilon_{k}+\theta\right)\right)-1+\mathcal{O}\left(\epsilon_{k}+\theta\right)=0,
\]
which we rearrange to get
\begin{equation}
\epsilon_{k}=-\frac{\omega_{k}}{\sin(\pi/\rho)}\theta+\mathcal{O}(\theta^{2}) \qquad (0\le k<\rho)\label{eq:epsilonsmalltheta}
\end{equation}
with $\omega_{k}=e^{(2k-1)\pi i/\rho}$. Using the estimate
\[
R(t_{k})=r-\sum_{-p_{-}<j\le p_{+}}\frac{t_{k}}{t_{k}-\tau_{j}}+\sum_{-q_{-}<j\le q_{+}}\frac{t_{k}}{t_{k}-\gamma_{j}}=\rho\frac{\sin(\pi/\rho)}{\omega_{k}\theta}+\mathcal{O}(1)
\]
together with 
\[
P(t_{k})=(-1)^{\rho}P^{(\rho)}(\tau_{1})\tau_{1}^{\rho}\omega_{k}^{\rho}\theta^{\rho}\sin(\pi/\rho)^{-\rho}\left(1+\mathcal{O}(\theta)\right),
\]
we conclude that the main term of \eqref{eq:Hmsmalltheta} is 
\begin{equation}
(-1)^{\rho+1}\frac{\sin(\pi/\rho)^{\rho-1}}{\rho P^{(\rho)}(\tau_{1})\theta^{\rho-1}\tau_{1}^{m+\rho}}\sum_{0\le k<\rho}\frac{\omega_{k}}{(1+\epsilon_{k})^{m}}.\label{eq:Hmasympsmalltheta}
\end{equation}

If $\theta\ge\delta/\sqrt{m}$ for some small $\delta,$ then \eqref{eq:epsilonsmalltheta}
implies that for each $2\le k<\rho$ 
\[
\frac{1}{|1+\epsilon_{k}|^{m}}\rightarrow0
\]
as $m\rightarrow\infty$ and thus the sign of \eqref{eq:Hmasympsmalltheta}
when $\theta=\theta_{h}$ is determined by the sign of the sum of
the first two terms which is the sign of $g(s_{h})$.

On the other hand, if $\theta<\delta/\sqrt{m}$ for $\delta\ll1$,
then the equation
\[
\Im\left(P(t_{1})R(t_{1})t_{1}^{m}\right)=0
\]
implies that 
\begin{align*}
0 & =\Im\left(\omega_{1}\left(1-\frac{\omega_{1}}{\sin(\pi/\rho)}\theta+\mathcal{O}\left(\theta^{2}\right)\right)^{-m}\right)\\
 & =\Im\left(\omega_{1}\exp\left(\frac{m\omega_{1}\theta}{\sin(\pi/\rho)}+\mathcal{O}\left(m\theta^{2}\right)\right)\right)
\end{align*}
from which we obtain the solutions
\[
\theta_{h}=\frac{h-1/\rho}{m}\pi+\mathcal{O}\left(\theta^{2}\right).
\]
We thus deduce that $g(s_{h})$ and $(-1)^{h+\rho+1}P^{(\rho)}(\tau_{1})$ carry the same sign. We claim that the sign of \eqref{eq:Hmasympsmalltheta} is also
$(-1)^{h+\rho+1}P^{(\rho)}(\tau_{1})$. Writing
\begin{align*}
(1+\epsilon_{k})^{m} & =\exp\left(-\frac{m\omega_{k}\theta}{\sin(\pi/\rho)}\right)\left(1+\mathcal{O}\left(\frac{h^{2}}{m}\right)\right),
\end{align*}
for large $m$, we see that the sign of \eqref{eq:Hmasympsmalltheta} is same as that of 
\begin{equation}
(-1)^{\rho+1}P^{(\rho)}(\tau_{1})\sum_{0\le k<\rho}\omega_{k}\exp\left(\frac{\omega_{k}(h-1/\rho)\pi}{\sin(\pi/\rho)}\right).\label{eq:expsumsmalltheta}
\end{equation}
Besides the factor $(-1)^{\rho+1}P^{(\rho)}(\tau_{1})$, the first
summand in \eqref{eq:expsumsmalltheta} is real and its sign is $(-1)^{h}$. The claim now follows from Proposition \ref{prop:signexppoly}.

\subsection*{Case 3: $\theta\rightarrow\pi/r$ as $m\rightarrow\infty$}

 If we define the angles $\theta^{*}=\pi/r-\theta$,
and $\theta_{j}^{*}$, $-p_{-}<j\le p_{+}$, and $\eta_{j}^{*}$,
$-q_{-}<j\le q_{+}$, by 
\[
\theta_{j}=\begin{cases}
\pi-\theta_{j}^{*} & \text{ if }0<j\le p_{+}\\
-\theta_{j}^{*} & \text{ if }-p_{-}<j\le0
\end{cases}\qquad\text{and}\qquad\eta_{j}=\begin{cases}
\pi-\eta_{j}^{*} & \text{if }1\le j\le q_{+}\\
-\eta_{j}^{*} & \text{if }-q_{-}<j\le0
\end{cases},
\]
then as a consequence of lemma \ref{lem:endlimits}, $\theta^*, \theta_j^*$ and $\eta_j^*$ all approach $0$ as $\theta\rightarrow\pi/r$. Using the equations \eqref{eq:taudef}
and \eqref{eq:etadef} we obtain 
\begin{eqnarray*}
\tau & = & \tau_{j}\frac{\sin\theta_{j}}{\sin(\theta_{j}-\theta)}\\
 & = & \tau_{j}\frac{\theta_{j}^{*}+\mathcal{O}(\theta_{j}^{*3})}{\sin(\pi/r)+\cos(\pi/r)(\theta_{j}^{*}-\theta^{*})+\mathcal{O}((\theta_{j}^{*}+\theta^{*})^{2})}\\
 & = & \frac{\tau_{j}\theta_{j}^{*}}{\sin(\pi/r)}\left(1-\cot\frac{\pi}{r}(\theta_{j}^{*}-\theta^{*})+\mathcal{O}((\theta_{j}^{*}+\theta^{*})^{2})\right),
\end{eqnarray*}
and 
\[
\tau=\frac{\gamma_{j}\eta_{j}^{*}}{\sin(\pi/r)}\left(1-\cot\frac{\pi}{r}(\eta_{j}^{*}-\theta^{*})+\mathcal{O}((\eta_{j}^{*}+\theta^{*})^{2})\right).
\]
We combine these identities with 
\[
\sum_{-p_{-}<j\le p_{+}}\theta_{j}^{*}-\sum_{-q_{-}<j\le q_{+}}\eta_{j}^{*}=r\theta^{*}
\]
to get 
\begin{equation}
\tau\sin\frac{\pi}{r}\sum_{-p_{-}<j\le p_{+}}\frac{1}{\tau_{j}}-\tau\sin\frac{\pi}{r}\sum_{-q_{-}<j\le q_{+}}\frac{1}{\gamma_{j}}=r\theta^{*}(1+\mathcal{O}(\theta_{j}^{*}+\eta_{j}^{*}+\theta^{*})).\label{eq:asympdifference}
\end{equation}
Thus, 
\begin{align*}
\theta_{k}^{*}\tau_{k}\left(\sum_{-p_{-}<j\le p_{+}}\frac{1}{\tau_{j}}-\sum_{-q_{-}<j\le q_{+}}\frac{1}{\gamma_{k}}\right) & =r\theta^{*}(1+\mathcal{O}(\theta^{*})), \quad \textrm{and}\\
\eta_{k}^{*}\gamma_{k}\left(\sum_{-p_{-}<j\le p_{+}}\frac{1}{\tau_{j}}-\sum_{-q_{-}<j\le q_{+}}\frac{1}{\gamma_{k}}\right) & =r\theta^{*}(1+\mathcal{O}(\theta^{*})).
\end{align*}

As $\theta\rightarrow\pi/r$, the polynomial in $t$ 
\[
\frac{P(t)}{P(\tau e^{i\theta})}-\frac{Q(t)}{Q(\tau e^{i\theta})}\left(\frac{t}{\tau e^{i\theta}}\right)^{r}
\]
has exactly $r$ zeros approaching the circle with radius $\tau$, each
of which satisfies $t_{k}/\tau\rightarrow e_{k}=e^{(2k-1)\pi i/r}$,
$0\le k<r$. This leave us to consider the sign of 
\begin{equation}
\sum_{0\le k<r}\frac{1}{P(t_{k})R(t_{k})t_{k}^{m}}.\label{eq:Hmlargetheta}
\end{equation}
Writing $t_{k}=\tau(e_{k}+\epsilon_{k})$, $\epsilon_{k}\in\mathbb{C}$,
we expand the left hand side of  
\[
\frac{P(t_{k})Q(\tau e^{i\theta})}{P(\tau e^{i\theta})Q(t_{k})}-\left(\frac{t_{k}}{\tau e^{i\theta}}\right)^{r}=0
\]
in a Taylor series centered at $\tau e^{i\theta}$ to get
\begin{equation} \label{eq:case3est}
1+\left(\frac{P'(0)}{P(0)}-\frac{Q'(0)}{Q(0)}\right)\left(e_{k}+\epsilon_{k}-e^{i\theta}\right)\tau+(e_{k}+\epsilon_{k})^{r}e^{ir\theta^{*}}=\mathcal{O}(\tau^{2}),
\end{equation}
where by \eqref{eq:asympdifference} 
\begin{equation*} 
\frac{P'(0)}{P(0)}-\frac{Q'(0)}{Q(0)}=-\sum_{-p_{-}<j\le p_{+}}\frac{1}{\tau_{j}}+\sum_{-q_{-}<j\le q_{+}}\frac{1}{\gamma_{j}}=-\frac{r\theta^{*}}{\tau\sin(\pi/r)}+\mathcal{O}(\theta^{*2}).
\end{equation*}
With 
\begin{eqnarray*}
(e_{k}+\epsilon_{k})^{r}e^{ir\theta^{*}}&=&-1-\frac{r\epsilon_{k}}{e_{k}}-ir\theta^{*}+\mathcal{O}\left(\epsilon_{k}^{2}+\theta^{*2}\right) \qquad \textrm{and}\\
\mathcal{O}(\tau^{2})&=&\mathcal{O}(\theta^{*2}),
\end{eqnarray*}
the equation \eqref{eq:case3est} can be rearranged as
 \[
-\frac{\theta^{*}}{\sin(\pi/r)}\left(e_{k}+\epsilon_{k}-e^{i\pi/r}\right)-\frac{\epsilon_{k}}{e_{k}}-i\theta^{*}=\mathcal{O}(\epsilon_{k}^{2}+\theta^{*2})
\]
from which we express $\epsilon_{k}$ as
\[
\epsilon_{k}=\frac{e_{k}\theta^{*}\left(\cos(\pi/r)-e_{k}\right)}{\sin(\pi/r)}+\mathcal{O}(\theta^{*2}) \qquad (0\le k<r).
\]
Using the estimates
\[
R(t_{k})=r-\sum_{-p_{-}<j\le p_{+}}\frac{t_{k}}{t_{k}-\tau_{j}}+\sum_{-q_{-}<j\le q_{+}}\frac{t_{k}}{t_{k}-\gamma_{j}}=r+\mathcal{O}(\theta^{*})
\]
and 
\[
P(t_{k})=P(0)(1+\mathcal{O}(\theta^{*})),
\]
we rewrite \eqref{eq:Hmlargetheta} as
\begin{equation} \label{eq:case3lastrewrite}
\frac{1}{rP(0)\tau^{m}}\sum_{0\le k<r}\frac{1}{(e_{k}+\epsilon_{k})^{m}}\left(1+\mathcal{O}(\theta^{*})\right).
\end{equation}
For the same reason as in Case 2, it suffices to consider $\theta^{*}<\delta/\sqrt{m}$
for small $\delta$, under which hypothesis the identity 
\[
\Im\left(P(t_{1})R(t_{1})t_{1}^{m}\right)=0
\]
gives 
\begin{align*}
0 & =\Im\left(e_{1}^{m}\left(1-i\theta^{*}+\mathcal{O}(\theta^{*2})\right)^{m}\right)\\
 & =\Im\left(e_{1}^{m}\exp\left(-i m\theta^{*}+\mathcal{O}(m\theta^{*2})\right)\right).
\end{align*}
Rearranging leads to
\begin{equation}
\theta_{h}^{*}=\frac{\pi}{r}-\frac{h\pi}{m}+\mathcal{O}(\theta^{*2})\label{eq:thetastarest}
\end{equation}
from which we deduce that the sign of $g(s_{h})$ agrees with that of $(-1)^{h}P(0)$.
In order to demonstrate that the sign of \eqref{eq:Hmlargetheta} is the same as that
of $(-1)^{h}P(0)$, we use a calculation analogous to that appearing at the end of Case 2 to write the main term of \eqref{eq:case3lastrewrite} as
\begin{equation}
\frac{e^{-m \theta^* \cot \pi/r}}{rP(0)\tau^{m}}\sum_{0\le k<r}e_{k}^{-m}\exp\left(\frac{m\theta^{*}e_{k}}{\sin(\pi/r)}\right)\label{eq:expsumlargetheta}.
\end{equation}
It thus remains to show that the sign of \eqref{eq:expsumlargetheta}  is $(-1)^{h}$ when $\theta^{*}=\theta_{h}^{*}$. This claim follows
from Proposition \ref{prop:signexppoly}, and completes Case 3, as well as the proof.
\end{proof}
\begin{lem} Let $H_m(z)$ be as in Theorem \ref{thm:maintheorem}, and let $z(\theta)$ be as in Definition \ref{def:zthetadef}.  If $2r \mid m$ then $H_{m}(z(\theta))$ has a zero on the interval $(\theta_{m/r-1},\pi/r)$.
\end{lem}
\begin{proof}
In case $2r|m$, the inequality $\theta_{h}^{*}>0$ implies
that the largest value of $h$ is $m/r-1$. Since the sign of \eqref{eq:expsumlargetheta}
is $(-1)^{\left\lfloor m/r\right\rfloor }$ as $\theta^{*}\rightarrow0^{+}$,
we conclude that $H_{m}(z(\theta))$ has a zero on the interval $(\theta_{m/r-1},\pi/r)$.
\end{proof}
\begin{lem} Let $H_m(z)$ be as in Theorem \ref{thm:maintheorem}, and let $z(\theta)$ be as in Definition \ref{def:zthetadef}. The sign of $H_{m}(z(\theta))$ is $(-1)^{p_{+}}$ as $\theta\rightarrow0$, and it is $(-1)^{p_{+}+\left\lfloor m/r\right\rfloor }$ and $\theta\rightarrow(\pi/r)-$.
\end{lem}
\begin{proof}
By evaluating the $(m\mod r)$th derivative of \eqref{eq:expsumlargetheta}
at $\theta^{*}=0$, we conclude that the sign of \eqref{eq:expsumlargetheta}
is $(-1)^{\left\lfloor m/r\right\rfloor }$ as $\theta^{*}\rightarrow0^{+}$ and
thus the sign of $H_{m}(z(\theta))$ as $\theta\rightarrow(\pi/r)^{-}$
is the sign of $(-1)^{\left\lfloor m/r\right\rfloor }P(0)$ which
is $(-1)^{\left\lfloor m/r\right\rfloor +p_{+}}$.\\
Finally, evaluating the $(\rho-1)$th derivative of 
\[
\sum_{0\le k<\rho}\omega_{k}\exp\left(\frac{m\omega_{k}\theta}{\sin(\pi/\rho)}\right)
\]
at $0$ we conclude that this expression is negative when
$\theta\rightarrow0$. Thus the sign of \eqref{eq:expsumsmalltheta}
is the sign of $(-1)^{\rho}P^{(\rho)}(\tau_{1})$ which is $(-1)^{p_{+}}$
by the product formula of $P(t)$.

\end{proof}

\end{document}